\documentclass[a4paper, 11pt]{article}
\usepackage[utf8]{inputenc}
\usepackage[english]{babel}
\usepackage{amssymb,amsmath,amsthm,amsfonts}

\usepackage{indentfirst}
\usepackage{comment}
\usepackage{enumerate}
\usepackage{enumitem}
\usepackage{framed}

\usepackage[pdftex]{graphicx}
\DeclareGraphicsExtensions{.pdf,.png,.jpg}

\usepackage{geometry}
\newtheorem{theorem}{Theorem}
\newtheorem{lemma}[theorem]{Lemma}

\newtheorem{remark}{Remark}

\DeclareMathOperator*{\argmin}{argmin}
\DeclareMathOperator*{\prox}{prox}

\DeclareMathOperator*{\proj}{proj}
\DeclareMathOperator*{\Id}{Id}

\usepackage{float} 
\usepackage{caption}
\usepackage{subcaption}

\usepackage[usenames]{color}
\usepackage{colortbl}

\allowdisplaybreaks

\title{A fast continuous time approach for nonsmooth convex optimization with time scaling and Tikhonov regularization}

\author{Robert Ernö Csetnek
\footnote{Faculty of Mathematics, University of Vienna, Oskar-Morgenstern-Platz 1, 1090 Vienna, Austria, 
{email: \tt robert.csetnek@univie.ac.at}. This work was supported by a grant of the Ministry of Research (Romania), Innovation and Digitization, CNCS-UEFISCDI, project number PN-III-P1-1.1-TE-2021-0138, within PNCDI III} \and
Mikhail A. Karapetyants 
\footnote{Faculty of Mathematics, University of Vienna, Oskar-Morgenstern-Platz 1, 1090 Vienna, Austria, 
{email: \tt mikhail.karapetyants@univie.ac.at.} Research supported by the Doctoral Programme \emph{Vienna Graduate School on Computational Optimization (VGSCO)} which is funded by FWF (Austrian Science Fund), project W 1260.} 
}

\begin{document}

\maketitle

\begin{abstract}
    
In a Hilbert setting we aim to study a second order in time differential equation, combining viscous and Hessian-driven damping, containing a time scaling parameter function and a Tikhonov regularization term. The dynamical system is related to the problem of minimization of a nonsmooth convex function. In the formulation of the problem as well as in our analysis we use the Moreau envelope of the objective function and its gradient and heavily rely on their properties. We show that there is a setting where the newly introduced system preserves and even improves the well-known fast convergence properties of the function and Moreau envelope along the trajectories and also of the gradient of Moreau envelope due to the presence of time scaling. Moreover, in a different setting we prove strong convergence of the trajectories to the element of minimal norm from the set of all minimizers of the objective. The manuscript concludes with various numerical results. 



\end{abstract} 

\begin{section}{Introduction}

In the Hilbert setting $H$, where $\langle \cdot, \cdot \rangle$ denotes the inner product and the norm is defined as usual $ \| \cdot \| = \sqrt{\langle \cdot, \cdot \rangle} $, we will study the convergence properties of the following second order in time differential equation
\begin{equation}\label{Syst}
    \ddot x(t) + \frac{\alpha}{t} \dot x(t) + \beta \frac{d}{dt} \nabla \Phi_{\lambda(t)}(x(t)) + b(t) \nabla \Phi_{\lambda(t)}(x(t)) + \varepsilon(t) x(t) = 0 \text{ for } t \geq t_0,
\end{equation}
with initial conditions $x(t_0) = x_0 \in H$, $\dot x(t_0) = \dot x_0 \in H$, where $ \alpha, \beta \text{ and } t_0 > 0 $, $\lambda: [t_0, +\infty) \mapsto \mathbb{R}_+$ and $b: [t_0, +\infty) \mapsto \mathbb{R}_+$ are non-negative, non-decreasing and differentiable, $ \Phi: H \mapsto \overline{\mathbb{R}} = \mathbb{R} \cup \{ \pm \infty \} $ is a proper, convex and lower semicontinuous function and $\Phi_\lambda$ is its Moreau envelope of the index $\lambda > 0$ and the function $\varepsilon: [t_0, +\infty) \mapsto \mathbb{R}_+$ is continuously differentiable and non-increasing with the property $\lim_{t \to +\infty} \varepsilon(t) = 0$. In addition, we assume that $\argmin \Phi$, which is the set of global minimizers of $\Phi$, is not empty and denote by $\Phi^*$ the optimal objective value of $\Phi$. The system \eqref{Syst} has a connection to the minimization problem
\[
\min_{x \in H} \Phi(x)
\]
of a proper, convex and lower semicontinuous function $\Phi$. Studying such systems provides better understanding of their discrete counterpart -- optimization algorithms, since there is a strong connection between them, and the question of transitioning from one to another attracts a lot of attention in the modern literature.

One of the main goals of this research is to improve (compared to \cite{BCL}) the fast rates of convergence for the Moreau envelope of the objective function and the objective function itself to $\Phi^*$, as well as for the gradient of the Moreau envelope of the objective function in terms of the Moreau parameter function $\lambda$ and the time scaling function $b$. Moreover, we also deduce the strong convergence of the trajectory of the dynamics to the minimal norm element of $\argmin \Phi$. We introduce two settings with different assumptions for each result. To conclude we provide multiple numerical results in order to illustrate our theoretical discoveries.

\subsection{Nonsmooth optimization with time scaling}

In the smooth setting the pioneering research in studying second order dynamical systems was conducted by Su-Boyd-Candes \cite{SBC} for the sake of obtaining faster asymptotic convergence for convex functions. They managed to deduce the rates of convergence of the function values being of the order $\frac{1}{t^2}$. Later Attouch-Peypouquet-Redont \cite{APR_0} also established the weak (and in some particular cases the strong) convergence of the trajectories to a minimizer of the objective function. In \cite{APR} the same authors continued the development in this direction by adding Hessian-driven damping term in order to obtain the rates for the gradient of the objective function and to eliminate any possible oscillations in the dynamical behaviour of the trajectories.

Concerning the nonsmooth setting we must point out that the Moreau envelope of a proper, convex and lower semicontinuous function $\Phi : H \to \overline{\mathbb{R}}$ proved to be of a significant importance in designing continuous-time approaches and numerical algorithms for the minimization of nonsmooth functions. The rigorous definition of this construction is
\begin{equation*}
\Phi_\lambda: H \to \mathbb{R}, \quad \Phi_{\lambda} (x) \ = \ \inf_{y \in H} \left\{ \Phi(y) + \frac{1}{2 \lambda} \| x - y \|^2 \right\},
\end{equation*}
where $\lambda > 0$ is the parameter of the Moreau envelope (see, for instance, \cite{BC}). One of the most important properties of Moreau approximation is that for every $\lambda > 0$,  the functions $\Phi$ and $\Phi_{\lambda}$ share the same optimal objective value and also the same set of minimizers. Moreover, $\Phi_\lambda$ is convex and continuously differentiable with
\begin{align}\label{Morprox}
\nabla \Phi_{\lambda} (x) = \frac{1}{\lambda} ( x - \prox\nolimits_{\lambda \Phi} (x)) \quad \forall x \in H,
\end{align}
and $\nabla \Phi_\lambda$ is $\frac{1}{\lambda}$-Lipschitz continuous, where
\[
\prox\nolimits_{\lambda \Phi}: H \to H, \quad \prox\nolimits_{\lambda \Phi} (x) = \argmin_{y \in H} \left\{ \Phi(y) + \frac{1}{2 \lambda} \| x - y \|^2 \right\},
\]
denotes the proximal operator of $\Phi$ of parameter $\lambda$. The last fact we would like to mention is that for every $x \in H$, the function $\lambda \in (0, +\infty) \to \Phi_\lambda (x)$ is nonincreasing  and differentiable (see \cite{AC}, Lemma A$1$), namely, 
\begin{equation}\label{Moreau_lambda_deriv}
    \frac{d}{d \lambda} \Phi_{\lambda} (x) = -\frac{1}{2} \| \nabla \Phi_{\lambda} (x) \|^2 \quad \forall \lambda > 0.
\end{equation}

Our research is a logical continuation of the one conducted in \cite{BK}, where authors applied the time rescaling technique to a nonsmooth optimization problem (for more information on time scaling see also \cite{ABCR_1, ACFR, ACR_0, ACR_2}). They considered the following system
\begin{equation}\label{Syst_old}
    \ddot x(t) + \frac{\alpha}{t} \dot x(t) + \beta(t) \frac{d}{dt} \nabla \Phi_{\lambda(t)}(x(t)) + b(t) \nabla \Phi_{\lambda(t)}(x(t)) = 0,
\end{equation}
where  $\alpha \geq 1 $, $t_0 > 0$, and $\beta: [t_0, +\infty) \mapsto [0, +\infty)$ and $ b,  \lambda: [t_0, +\infty) \mapsto (0, +\infty)  $ are differentiable functions. On the one hand, the presence of the Hessian damping term is believed to help reducing the oscillations in the dynamical behaviour and provides the rates for the gradient of the objective function $\Phi$. On the other hand, the time-scaling technique (which is considered to be an artificial way to speed up the convergence of values) affects the convergence rates while bringing more restrictions to the analysis. The following properties were established
\[
\Phi_{\lambda(t)}(x(t)) - \Phi^* = o\left( \frac{1}{t^2 b(t)} \right) \text{ and } \| \dot x(t) \| = o\left( \frac{1}{t} \right) \text{ as } t \to +\infty,
\]
from where through proximal mapping the convergence rates for the objective function $\Phi$ itself along the trajectory were obtained
\[
\Phi \big( \prox\nolimits_{\lambda(t) \Phi} (x(t)) \big) - \Phi^* = o\left( \frac{1}{t^2 b(t)} \right) \text{ and } \| \prox\nolimits_{\lambda(t) \Phi} (x(t)) - x(t) \| = o\left( \frac{\sqrt{\lambda(t)}}{t \sqrt{b(t)}} \right) \text{ as } t \to +\infty.
\]
Note that by taking $b(\cdot) \equiv 1$ we arrive at the well-known convergence rate of the values being of the order $o\left( \frac{1}{t^2} \right)$. In addition, the following rates for the gradient of the Moreau envelope were deduced
\[
\| \nabla \Phi_{\lambda(t)} (x(t)) \| \ = \ o \left( \frac{1}{t \sqrt{b(t) \lambda(t)}} \right), \text{ as } t \to +\infty .
\]
Finally, the weak convergence of the trajectories $x(t)$ to a minimizer of $\Phi$ as $t \to +\infty$ was obtained.

In our analysis we borrow some ideas of \cite{BK} and develop them further in order to fit the new setting, namely, to adapt to a presence of the whole new term -- Tikhonov regularization. The analysis becomes more involved and technical, some fundamental properties of Tikhonov regularization had to be proved for a nonsmooth setting. Its presence affects the set of conditions, which we have to impose on the system parameters: even though some of the conditions are formulated in the same spirit as in \cite{BK} (for instance, \eqref{A_6_0} and \eqref{tbt}), the other ones are completely new due to the presence of the Tikhonov term. Moreover, depending on how fast $\varepsilon$ decays, two different setting arise providing different fundamental results (Sections $3$ and $4$).

\subsection{Tikhonov regularization}

It turned out that having additional term with specific properties in a system equation leads to improving the weak convergence of the trajectories to a minimizer of the objective function $\Phi$ to a strong one to the element of minimal norm of $\argmin \Phi$. Such systems were studied, for instance, in \cite{ABCR, ABCR_0, ACR, ACR_1, AL_0, BCL, L}. The main goal of such a research is to show that these systems preserve all the typical properties of the second order in time dynamical system (fast convergence of the values, the rates for the gradient etc.) but moreover there is an improvement to the strong convergence of the trajectories to the minimal norm solution instead of a weak one to an arbitrary minimizer. One of the many examples of such systems is presented below (see  \cite{BCL})
\begin{equation*}
    \ddot x(t) + \frac{\alpha}{t} \dot x(t) + \beta \nabla^2 \Phi(x(t)) \dot x(t) + \nabla \Phi(x(t)) + \varepsilon(t) x(t) = 0 \text{ for } t \geq t_0,
\end{equation*}
where $ \alpha \geq 3$, $t_0 > 0$, $\Phi: H \mapsto \mathbb{R}$ is twice continuously differentiable and convex and for the rest of the section the function $\varepsilon: [t_0, +\infty) \mapsto \mathbb{R}_+$ is continuously differentiable and non-increasing with the property $\lim_{t \to +\infty} \varepsilon(t) = 0$. In that manuscript they provided two settings: one for the fast convergence of values obtaining
\[
\Phi(x(t)) - \Phi^* \ = \ o\left( \frac{1}{t^2} \right), \text{ as } t \to +\infty
\]
and the weak convergence of the trajectories to a minimizer of $\Phi$ and another setting for the strong convergence of $x$ to $x^*$, as $t \to +\infty$.

Another fine example is given in \cite{ABCR}:
\begin{equation}\label{syst_tikh}
    \ddot x(t) + \alpha \sqrt{\varepsilon(t)} \dot x(t) + \nabla \Phi(x(t)) + \varepsilon(t) x(t) = 0 \text{ for } t \geq t_0,
\end{equation}
where $ \alpha $, $ t_0 > 0 $ and $\Phi: H \mapsto \mathbb{R}$ is continuously differentiable and convex. In that paper authors obtained the rates for the function values $ \Phi(x(t)) - \Phi^* $, as well as for the quantity $ \| x(t) - x_{\varepsilon(t)} \| $, as $t \to +\infty$, where $ x_{\varepsilon(t)} = \argmin_H \left( \Phi(x) + \frac{\varepsilon(t) \| x \|^2}{2} \right) $. Thus, they assured the strong convergence of the trajectories to the minimal norm solution $ x^* = \proj_{\argmin \Phi}(0) $ under the appropriate assumptions and properly chosen energy functional, using the properties of Tikhonov regularization. The most important thing about this approach is that authors were able to establish fast convergence of values and strong convergence of the trajectories in the very same setting.

The next step was done in \cite{ABCR_0}:
\begin{equation*}
    \ddot x(t) + \alpha \sqrt{\varepsilon(t)} \dot x(t) + \beta \frac{d}{dt} \Big( \nabla \varphi_t (x(t)) + (p - 1) \varepsilon(t) x(t) \Big) + \nabla \varphi_t (x(t)) = 0 \text{ for } t \geq t_0,
\end{equation*}
where $\varphi_t (x) = \Phi(x) + \frac{\varepsilon(t) \| x \|^2}{2}$, $ \Phi: H \mapsto \mathbb{R}$ is twice continuously differentiable and convex and $p \in [0, 1]$. This system while preserving all the properties of \eqref{syst_tikh}, additionally provides the integral estimate for the norm of the gradient of $\varphi_t$.

\subsection{Our contribution}

In that paper we will develop the ideas presented in \cite{BCL} to cover the nonsmooth case with time scaling. We will obtain the fast convergence of the function values (as well as for the gradient of the Moreau envelope of the objective fucntion $\Phi$) for the family of dynamical systems \eqref{Syst} governed by the Moreau envelope of the nonsmooth function $\Phi$ and having the Tiknonov term in their formulation:
\[
\Phi_{\lambda(t)} (x(t)) - \Phi^* \ = \ o \left( \frac{1}{t^2 b(t)} \right) \text{ as } t \to +\infty;
\]
 in terms of the function itself:
\[
\Phi(\prox\nolimits_{\lambda(t) \Phi}(x(t))) - \Phi^* \ = \ o \left( \frac{1}{t^2 b(t)} \right) \text{ as } t \to +\infty,
\]
where
\[
\| \prox\nolimits_{\lambda(t) \Phi} (x(t)) - x(t) \| \ = \ o \left( \frac{\sqrt{\lambda(t)}}{t \sqrt{b(t)}} \right) \text{ as } t \to +\infty
\]
and finally
\[
\| \nabla \Phi_{\lambda(t)} (x(t)) \| \ = \ o \left( \frac{1}{t \sqrt{b(t) \lambda(t)}} \right) \text{ as } t \to +\infty.
\]
We will also deduce (under some appropriate conditions) the following result
\[
\liminf_{t \to +\infty} \| x(t) - x^* \| = 0,
\]
which under some restrictions will be improved to the full strong convergence of the trajectories of \eqref{Syst} to the minimal norm solution.

The paper is organized in the following way. Section $2$ is devoted to some preliminary results, which we will need later. We will establish the fast rates of convergence of function values and its Moreau envelope, as well as the gradient of Moreau envelope along the trajectories of the dynamical system (Section $3$). We will show that under some assumptions the strong convergence of the trajectories to the element of minimal norm from the set of all minimizers of the objective function takes place (Section $4$). We will provide two settings for the polynomial choice of parameter functions to fulfill the assumptions made through the analysis (Section $5$) and equip this manuscript with various numerical results (Section $6$).

\end{section}

\begin{section}{Preparatory results}

We start with the following lemma (see \cite{BC}, Proposition 12.22, for the first term of the lemma and \cite{AP}, Appendix, A$1$, for the second one).
\begin{lemma}\label{L_0}

Let $\Phi: H \mapsto \overline{\mathbb{R}}$ be a proper, convex and lower semicontinuous function, $\lambda, \mu > 0$. Then
\begin{enumerate}

    \item\label{i} $(\Phi_\lambda)_\mu = \Phi_{\lambda + \mu}$.
    
    \item\label{ii} $ \prox_{\mu \Phi_\lambda} = \frac{\lambda}{\lambda + \mu} \Id + \frac{\mu}{\lambda + \mu} \prox_{(\lambda + \mu)\Phi} $.
    
\end{enumerate}

\end{lemma}

Let us mention two key properties of the Tikhonov regularization, which we will use later in the analysis (see, for instance, \cite{A} or \cite{BC} Theorem 23.44 for its classic analogue). First let us introduce the strongly convex function $\varphi_{\varepsilon(t), \lambda(t)}: H \mapsto \mathbb{R}$ as $\varphi_{\varepsilon(t), \lambda(t)} (x) = \Phi_{\lambda(t)}(x) + \frac{\varepsilon(t) \| x \|^2}{2}$ and denote the unique minimizer of $\varphi_{\varepsilon(t), \lambda(t)}$ as $x_{\varepsilon(t), \lambda(t)} = \argmin_{H} \varphi_{\varepsilon(t), \lambda(t)}$. Thus, the first order optimality condition reads as
\begin{equation}\label{FOOC}
    \nabla \Phi_{\lambda(t)} (x_{\varepsilon(t), \lambda(t)}) + \varepsilon(t) x_{\varepsilon(t), \lambda(t)} = 0.
\end{equation}
Now we are ready to formulate the following result:

\begin{lemma}

Suppose that 
\begin{equation}\label{A_0}
    \lim_{t \to +\infty} \lambda(t) \varepsilon(t) \ = \ 0.
\end{equation}
Then the following properties of the mapping $t \mapsto x_{\varepsilon(t), \lambda(t)}$ are satisfied:
\begin{equation}\label{T_0}
    \text{ for } x^* = \proj_{\argmin \Phi}(0), \ \| x_{\varepsilon(t), \lambda(t)} \| \leq \| x^* \| \text{ for all } t \geq t_0
\end{equation}
and
\begin{equation}\label{T_1}
    \lim_{t \to +\infty} \| x_{\varepsilon(t), \lambda(t)} - x^* \| = 0.
\end{equation}

\end{lemma}

\begin{proof}

By the monotonicity of $\nabla \Phi_\lambda$ we deduce
\[
\left\langle \nabla \Phi_{\lambda(t)}(x_{\varepsilon(t), \lambda(t)}) - \nabla \Phi_{\lambda(t)}(x^*), x_{\varepsilon(t), \lambda(t)} - x^* \right\rangle \ \geq \ 0.
\]
By \eqref{FOOC} we obtain
\[
\left\langle - \varepsilon(t) x_{\varepsilon(t), \lambda(t)}, x_{\varepsilon(t), \lambda(t)} - x^* \right\rangle \ = \ \varepsilon(t) \left( - \| x_{\varepsilon(t), \lambda(t)} \|^2 + \left\langle x_{\varepsilon(t), \lambda(t)}, x^* \right\rangle \right) \ \geq \ 0.
\]
Using Cauchy-Schwarz inequality we derive
\[
\| x_{\varepsilon(t), \lambda(t)} \| \ \leq \ \| x^* \|.
\]
This proves the first claim. For the second one consider \eqref{FOOC} again and note that it is equivalent to 
\[
x_{\varepsilon(t), \lambda(t)} = \prox\nolimits_{\frac{1}{\varepsilon(t)}\Phi_{\lambda(t)}} (0) = \frac{\prox\nolimits_{\left( \lambda(t) + \frac{1}{\varepsilon(t)} \right)\Phi} (0)}{\lambda(t) \varepsilon(t) + 1}
\]
by the item \ref{ii} of Lemma \ref{L_0}. Note that $\lambda(t) + \frac{1}{\varepsilon(t)} \to +\infty$, as $t \to +\infty$. Thus, the rest of the proof goes in line with Theorem 23.44 of \cite{BC}.

\end{proof}

Our nearest goal is to deduce the existence and uniqueness of the solutions of the dynamical system \eqref{Syst}. Suppose $\beta > 0$. Let us integrate \eqref{Syst} from $t_0$ to $t$ to obtain
\[
    \dot x(t) + \beta \nabla \Phi_{\lambda(t)} (x(t)) + \int_{t_0}^t \left( \frac{\alpha}{s} \dot x(s) + b(s) \nabla \Phi_{\lambda(s)} (x(s)) + \varepsilon(s) x(s) \right) ds - \left( \dot x(t_0) + \beta \nabla \Phi_{\lambda(t_0)} (x(t_0)) \right) \ = \ 0.
\]
Denoting $z(t) := \int_{t_0}^t \left( \frac{\alpha}{s}  \dot x(s) + b(s) \nabla \Phi_{\lambda(s)} (x(s)) + \varepsilon(s) x(s) \right) ds - \big( \dot x(t_0) + \beta \nabla \Phi_{\lambda(t_0)} (x_0)) \big)$ for every $t \geq t_0$ and noticing that $\dot z(t) = \frac{\alpha}{t}\dot x(t) +b(t) \nabla \Phi_{\lambda(t)} (x(t)) + \varepsilon(t) x(t)$ we deduce, that \eqref{Syst} is equivalent  to
\begin{equation*}
\begin{cases}
    &\dot x(t) + \beta \nabla \Phi_{\lambda(t)} (x(t)) + z(t) = 0, \\
    &\dot z(t) - \frac{\alpha}{t} \dot x(t) - b(t) \nabla \Phi_{\lambda(t)}(x(t)) - \varepsilon(t) x(t) = 0, \\
    &x(t_0) = x_0, \ z(t_0) = -\left(\dot x(t_0) + \beta \nabla \Phi_{\lambda(t_0)} (x_0) \right).
\end{cases}
\end{equation*}
Let us multiply the first line by the function $b$ and the second one by the constant $\beta$ and then sum them up to get rid of the gradient of the Moreau envelope in the second equation
\begin{equation*}
\begin{cases}
    &\dot x(t) + \beta \nabla \Phi_{\lambda(t)} (x(t)) + z(t) = 0, \\
    &\beta \dot z(t) + \left( b(t) - \frac{\alpha \beta}{t} \right) \dot x(t) - \beta \varepsilon(t) x(t) + b(t) z(t) = 0,\\
    &x(t_0) = x_0, \ z(t_0) = -\left(\dot x(t_0) + \beta \nabla \Phi_{\lambda(t_0)} (x_0) \right).
\end{cases}
\end{equation*}
We denote now $y(t) = \beta z(t) + \left( b(t) - \frac{\alpha \beta}{t} \right) x(t)$, and, after simplification, we obtain the following equivalent formulation for the dynamical system 
\begin{equation*}
\begin{cases}
    &\dot x(t) + \beta \nabla \Phi_{\lambda(t)} (x(t)) + \left( \frac{\alpha}{t} - \frac{b(t)}{\beta} \right) x(t) + \frac{1}{\beta} y(t) = 0, \\
    &\dot y(t) - \left( \dot b(t) + \frac{\alpha \beta}{t^2} + \beta \varepsilon(t) + \frac{b^2(t)}{\beta} - \frac{\alpha b(t)}{t} \right) x(t) + \frac{b(t)}{\beta} y(t) = 0, \\
    &x(t_0) = x_0, \ y(t_0) = -\beta \left( \dot x(t_0) + \beta \nabla \Phi_{\lambda(t_0)} (x_0) \right) + \left( b(t_0) - \frac{\alpha \beta}{t_0} \right) x_0.
\end{cases}
\end{equation*}
In case $\beta = 0$ for every $t \geq t_0$, \eqref{Syst}  can be equivalently written as
\begin{equation*}
\begin{cases}
&\dot x(t)  - y(t) = 0, \\
&\dot y(t) + \frac{\alpha}{t} y(t) + b(t) \nabla \Phi_{\lambda(t)}(x(t)) + \varepsilon(t) x(t) = 0, \\
&x(t_0) = x_0, \ y(t_0) = \dot x(t_0).
\end{cases}
\end{equation*}
Based on the two reformulations of the dynamical system \eqref{Syst} we formulate the following existence and uniqueness result, which is a consequence of Cauchy-Lipschitz theorem for strong global solutions. The result can be proved in the lines of the proofs of Theorem $1$ in \cite{AL} or of Theorem $1.1$ in \cite{APR} with some small adjustments.

\begin{theorem}

Suppose that there exists $\lambda_0 > 0$ such that $\lambda(t) \geq \lambda_0$ for all $t \geq t_0$. Then for every $(x_0, \dot x(t_0)) \in H \cdot H $ there exists a unique strong global solution $x: [t_0, +\infty) \mapsto H$ of the continuous dynamics \eqref{Syst} which satisfies the Cauchy initial conditions $x(t_0) = x_0$ and $\dot x(t_0) = \dot x_0$.

\end{theorem}

\end{section}

\begin{section}{Fast convergence rates of the function and Moreau envelope values}

This chapter is devoted to obtaining the rates of convergence for the Moreau envelope values and for the values of function $\Phi$ itself. We will heavily rely on the tools and techniques provided by the Lyapunov analysis. We introduce a slightly modified energy function from \cite{BCL}. For $2 \leq q \leq \alpha - 1$ we define
\begin{equation}\label{Energy_0}
\begin{split}
    E_q(t) \ = \ &(t^2 b(t) - \beta (q + 2 - \alpha)t) \left( \Phi_{\lambda(t)} (x(t)) - \Phi^* \right) + \frac{t^2 \varepsilon(t)}{2} \| x(t) \|^2 \\
    &+ \ \frac{1}{2} \| q(x(t) - x^*) + t (\dot x(t) + \beta \nabla \Phi_{\lambda(t)} (x(t)) \|^2 + \frac{q (\alpha - 1 - q)}{2} \| x(t) - x^* \|^2.
\end{split}
\end{equation}

The key assumptions which are essential to our analysis are the following: for all $t \geq t_0$
\begin{framed}
\begin{equation}\label{A_6_0}
    (\alpha - 3) t b(t) - t^2 \dot b(t) + \beta (2 - \alpha) \ \geq \ 0,
\end{equation}
\begin{equation}\label{A_9}
    \exists a \geq 1 \text{ such that } \  2 \dot \varepsilon(t) \ \leq \ - a \beta \varepsilon^2(t),
\end{equation}
\begin{equation}\label{integrability_T}
    \int_{t_0}^{+\infty} t \varepsilon(t) dt \ < \ +\infty,
\end{equation}
\begin{equation}\label{tbt}
    b(t_0) \ \geq \ \frac{\beta}{t_0} \text{ and } b(t_0) > \frac{1}{a}
\end{equation}
and
\begin{equation}\label{A_6}
    \exists \delta: 0 < \delta < \alpha - 3 \text{ such that } (\alpha - 3) t b(t) - t^2 \dot b(t) + \beta (2 - \alpha) \ \geq \ \delta t b(t).
\end{equation}
\end{framed}

\begin{theorem}\label{Th_0}

Suppose $\alpha \geq 3$ and assume that \eqref{A_6_0}, \eqref{A_9}, \eqref{integrability_T}, \eqref{tbt} hold for all $t \geq t_0$. Then
\[
\Phi_{\lambda(t)} (x(t)) - \Phi^* \ = \ O\left( \frac{1}{t^2 b(t)} \right), \text{ as } t \to +\infty,
\]
\[
\| \dot x(t) + \beta \nabla \Phi_{\lambda(t)} (x(t)) \| \ = \ O\left( \frac{1}{t} \right), \text{ as } t \to +\infty.
\]
Moreover, one has for all $a \geq 1$
\begin{align*}
    &t \varepsilon(t) \| x(t) - x^* \|^2, \ t \varepsilon(t) \| x(t) \|^2, \ \left( (\alpha - 3) t b(t) - t^2 \dot b(t) + \beta (2 - \alpha) \right) \left( \Phi_{\lambda(t)} (x(t)) - \Phi^* \right) \text{ and } \\
    &\left( \left( t^2 b(t) - \beta t \right) \frac{\dot \lambda(t)}{2} - \beta^2 t + \beta t^2 \left( b(t) - \frac{1}{a} \right) \right) \| \nabla \Phi_{\lambda(t)} (x(t)) \|^2 \in L^1 \big( [t_0, +\infty), \mathbb{R} \big).
\end{align*}
If, in addition, $\alpha > 3$ and \eqref{A_6} holds, then the trajectory $x$ is bounded and 
\[
\int_{t_0}^{+\infty} t \| \dot x(t) \|^2 dt \ < \ +\infty
\]
and
\[
\int_{t_0}^{+\infty} t b(t) \left( \Phi_{\lambda(s)} (x(s)) - \Phi^* \right) \ < \ +\infty.
\]

\end{theorem}

\begin{proof}

Let us compute the time derivative of the energy function. For every $t \geq t_0$ using \eqref{Moreau_lambda_deriv} we derive
\begin{align*}
    \dot E_q(t) \ = \ &\left( 2t b(t) + t^2 \dot b(t) - \beta (q + 2 - \alpha) \right) \left( \Phi_{\lambda(t)} (x(t)) - \Phi^* \right) + t^2 \varepsilon(t) \langle x(t), \dot x(t) \rangle \\
    &+ \ q (\alpha - 1 - q) \langle \dot x(t), x(t) - x^* \rangle \\
    &+ \ \left( t^2 b(t) - \beta (q + 2 - \alpha)t \right) \left( \langle \nabla \Phi_{\lambda(t)} (x(t)), \dot x(t) \rangle - \frac{\dot \lambda(t)}{2} \| \nabla \Phi_{\lambda(t)} (x(t)) \|^2 \right) \\
    &+ \ \frac{2t \varepsilon(t) + t^2 \dot \varepsilon(t)}{2} \| x(t) \|^2 \\
    &+ \ \Bigg\langle q(x(t) - x^*) + t (\dot x(t) + \beta \nabla \Phi_{\lambda(t)} (x(t)), (q + 1) \dot x(t) + \beta \nabla \Phi_{\lambda(t)} (x(t)) \\
    &+ \ t \left( \ddot x(t) + \beta \frac{d}{dt} \nabla \Phi_{\lambda(t)} (x(t)) \right) \Bigg\rangle.
\end{align*}
Define $v(t) = q(x(t) - x^*) + t (\dot x(t) + \beta \nabla \Phi_{\lambda(t)} (x(t))$. Using \eqref{Syst} to replace $\ddot x(t) + \beta \frac{d}{dt} \nabla \Phi_{\lambda(t)} (x(t))$ we obtain
\begin{align*}
    \langle v(t), \dot v(t) \rangle \ = \ &\Bigg\langle q(x(t) - x^*) + t (\dot x(t) + \beta \nabla \Phi_{\lambda(t)} (x(t)), (q + 1 - \alpha) \dot x(t) \\
    &+ \ \left( \beta - t b(t) \right) \nabla \Phi_{\lambda(t)} (x(t)) - t \varepsilon(t) x(t) \Bigg\rangle \\
    = \ &q (q + 1 - \alpha) \left\langle x(t) - x^*, \dot x(t) \right\rangle + (q + 1 - \alpha) t \| \dot x(t) \|^2 \\
    &+ \ \left( \beta (q + 2 - \alpha)t - t^2 b(t) \right) \left\langle \nabla \Phi_{\lambda(t)} (x(t)), \dot x(t) \right\rangle + \left( \beta^2 t - \beta t^2 b(t) \right) \| \nabla \Phi_{\lambda(t)} (x(t)) \|^2 \\
    &- \ t^2 \varepsilon(t) \langle x(t), \dot x(t) \rangle - \beta t^2 \varepsilon(t) \langle x(t), \nabla \Phi_{\lambda(t)} (x(t)) \rangle \\
    &- \ q t \left\langle \left( b(t) - \frac{\beta}{t} \right) \nabla \Phi_{\lambda(t)} (x(t)) + \varepsilon(t) x(t), x(t) - x^* \right\rangle.
\end{align*}
By \eqref{tbt} one has $b(t) - \frac{\beta}{t} > 0$ for all $t \geq t_0$, and thus for a strongly convex function $\varphi_t (x) \ = \ \left( b(t) - \frac{\beta}{t} \right) \Phi_{\lambda(t)} (x) + \frac{\varepsilon(t)}{2} \| x \|^2$ we have
\[
\varphi_t(x^*) - \varphi_t(x) \ \geq \ \left\langle \nabla \varphi_t(x), x^* - x \right\rangle + \frac{\varepsilon(t)}{2} \| x^* - x \|^2
\]
or
\begin{equation*}\label{Str_conv_0}
\begin{split}
    &- q t \left\langle \left( b(t) - \frac{\beta}{t} \right) \nabla \Phi_{\lambda(t)} (x(t)) + \varepsilon(t) x(t), x(t) - x^* \right\rangle \\ \leq \ &- q t \left( b(t) - \frac{\beta}{t} \right) \left( \Phi_{\lambda(t)} (x(t)) - \Phi^* \right) - q t \frac{\varepsilon(t)}{2} \| x(t) \|^2 - q t \frac{\varepsilon(t)}{2} \| x(t) - x^* \|^2 + q t \frac{\varepsilon(t)}{2} \| x^* \|^2.
\end{split}
\end{equation*}
Therefore, for every $t \geq t_0$
\begin{equation*}
\begin{split}
    \dot E_q(t) \ \leq \ &\left( (2 - q) t b(t) + t^2 \dot b(t) - \beta (2 - \alpha) \right) \left( \Phi_{\lambda(t)} (x(t)) - \Phi^* \right) + (q + 1 - \alpha) t \| \dot x(t) \|^2 \\
    &- \ \left( \left( t^2 b(t) - \beta (q + 2 - \alpha)t \right) \frac{\dot \lambda(t)}{2} - \beta^2 t + \beta t^2 b(t) \right) \| \nabla \Phi_{\lambda(t)} (x(t)) \|^2 \\
    &+ \ \frac{(2 - q)t \varepsilon(t) + t^2 \dot \varepsilon(t)}{2} \| x(t) \|^2 - q t \frac{\varepsilon(t)}{2} \| x(t) - x^* \|^2 + q t \frac{\varepsilon(t)}{2} \| x^* \|^2 \\
    &- \ \beta t^2 \varepsilon(t) \langle x(t), \nabla \Phi_{\lambda(t)} (x(t)) \rangle.
\end{split}
\end{equation*}
Notice that for $a \geq 1$
\[
-\beta t^2 \varepsilon(t) \langle x(t), \nabla \Phi_{\lambda(t)} (x(t)) \rangle \ \leq \ \frac{\beta t^2}{a} \| \nabla \Phi_{\lambda(t)} (x(t)) \|^2 + \frac{a \beta t^2 \varepsilon^2(t)}{4} \| x(t) \|^2,
\]
which leads to
\begin{equation}\label{Energy_Pre}
\begin{split}
    \dot E_q(t) \ \leq \ &\left( (2 - q) t b(t) + t^2 \dot b(t) - \beta (2 - \alpha) \right) \left( \Phi_{\lambda(t)} (x(t)) - \Phi^* \right) + (q + 1 - \alpha) t \| \dot x(t) \|^2 \\
    &- \ \left( \Big( t^2 b(t) - \beta (q + 2 - \alpha)t \Big) \frac{\dot \lambda(t)}{2} - \beta^2 t + \beta t^2 \left( b(t) - \frac{1}{a} \right) \right) \| \nabla \Phi_{\lambda(t)} (x(t)) \|^2 \\
    &+ \ \frac{2 (2 - q) t \varepsilon(t) + 2 t^2 \dot \varepsilon(t) + a \beta t^2 \varepsilon^2(t)}{4} \| x(t) \|^2 - q t \frac{\varepsilon(t)}{2} \| x(t) - x^* \|^2 + q t \frac{\varepsilon(t)}{2} \| x^* \|^2
\end{split}
\end{equation}
for every $t \geq t_0$. Note that $b(t) - \frac{1}{a} > 0$ for all $t \geq t_0$. Then, due to the properties of $b$, there exists $t^* \geq t_0$ such that $t^2 b(t) - \beta (q + 2 - \alpha)t \ \geq \ 0$ for all $t \geq t^*$ and all $q \in (2, \alpha - 1]$. Therefore, since  $\dot \lambda(t) \geq 0$ for all $t \geq t_0$, there exists $t^{**}$, namely, $t^{**} = \max \left\{ t^*, \frac{\beta}{b(t_0) - \frac{1}{a}} \right\}$, such that
\[
\left( \Big( t^2 b(t) - \beta (q + 2 - \alpha)t \Big) \frac{\dot \lambda(t)}{2} - \beta^2 t + \beta t^2 \left( b(t) - \frac{1}{a} \right) \right) \ \geq \ 0 \text{ for all } t \geq t^{**}.
\]
Consider now two cases with $t \geq t^{**}$. First, take $q = \alpha - 1$ to obtain from \eqref{Energy_Pre}
\begin{equation}\label{Energy_Pre_0}
\begin{split}
    \dot E_{\alpha - 1}(t) \ \leq \ &\left( (3 - \alpha) t b(t) + t^2 \dot b(t) - \beta (2 - \alpha) \right) \left( \Phi_{\lambda(t)} (x(t)) - \Phi^* \right) \\
    &- \ \left( \Big( t^2 b(t) - \beta t \Big) \frac{\dot \lambda(t)}{2} - \beta^2 t + \beta t^2 \left( b(t) - \frac{1}{a} \right) \right) \| \nabla \Phi_{\lambda(t)} (x(t)) \|^2 \\
    &+ \ \frac{2 (3 - \alpha) t \varepsilon(t) + 2 t^2 \dot \varepsilon(t) + a \beta t^2 \varepsilon^2(t)}{4} \| x(t) \|^2 - (\alpha - 1) t \frac{\varepsilon(t)}{2} \| x(t) - x^* \|^2 \\
    &+ \ (\alpha - 1) t \frac{\varepsilon(t)}{2} \| x^* \|^2
\end{split}
\end{equation}
for every $t \geq t_0$. Under the assumptions \eqref{A_6_0} and \eqref{A_9} we conclude starting from $t^{**}$ that
\begin{equation}\label{Energy_FF_0}
    \dot E_{\alpha - 1}(t) \ \leq \ \frac{(\alpha - 1) t \varepsilon(t)}{2} \| x^* \|^2.
\end{equation}
Under the assumption \eqref{integrability_T} using the fact that $t \mapsto E_{\alpha - 1}(t)$ is bounded from below we deduce the existence of the limit $\lim_{t \to +\infty} E_{\alpha - 1}(t)$ due to the Lemma \ref{L_1} and, therefore, $t \mapsto E_{\alpha - 1}(t)$ is bounded, which leads to
\[
\Phi_{\lambda(t)} (x(t)) - \Phi^* \ = \ O\left( \frac{1}{t^2 b(t)} \right), \text{ as } t \to +\infty.
\]
From the boundedness of $t \mapsto \| (\alpha - 1)(x(t) - x^*) + t (\dot x(t) + \beta \nabla \Phi_{\lambda(t)} (x(t)) \|^2$ we obtain
\[
\| \dot x(t) + \beta \nabla \Phi_{\lambda(t)} (x(t)) \| \ = \ O\left( \frac{1}{t} \right), \text{ as } t \to +\infty,
\]
using the following inequality, which is true for every $t \geq t_0$
\[
t^2 \| \dot x(t) + \beta \nabla \Phi_{\lambda(t)} (x(t)) \|^2 \ \leq \ 2\| (\alpha - 1)(x(t) - x^*) + t (\dot x(t) + \beta \nabla \Phi_{\lambda(t)} (x(t)) \|^2 + 2 (\alpha - 1)^2 \| x(t) - x^* \|^2.
\]
Moreover, integrating \eqref{Energy_Pre_0} one may obtain the integrability of $t \varepsilon(t) \| x(t) - x^* \|^2$ as well as the other terms in \eqref{Energy_Pre_0}.
Consider now $q = \alpha - 1 - \delta$, where $\delta$ is defined by \eqref{A_6}. Thus, \eqref{Energy_Pre} becomes
\begin{equation}\label{Energy_Pre_1}
\begin{split}
    \dot E_{\alpha - 1 - \delta}(t) \ \leq \ &\left( (3 - \alpha + \delta) t b(t) + t^2 \dot b(t) - \beta (2 - \alpha) \right) \left( \Phi_{\lambda(t)} (x(t)) - \Phi^* \right) - \delta t \| \dot x(t) \|^2 \\
    &- \ \left( \Big( t^2 b(t) - \beta (1 - \delta)t \Big) \frac{\dot \lambda(t)}{2} - \beta^2 t + \beta t^2 \left( b(t) - \frac{1}{a} \right) \right) \| \nabla \Phi_{\lambda(t)} (x(t)) \|^2 \\
    &+ \ \frac{2 (3 - \alpha + \delta) t \varepsilon(t) + 2 t^2 \dot \varepsilon(t) + a \beta t^2 \varepsilon^2(t)}{4} \| x(t) \|^2 - (\alpha - 1 - \delta) t \frac{\varepsilon(t)}{2} \| x(t) - x^* \|^2 \\
    &+ \ (\alpha - 1 - \delta) t \frac{\varepsilon(t)}{2} \| x^* \|^2.
\end{split}
\end{equation}
Under the assumptions \eqref{A_9} and \eqref{A_6} we deduce $\dot E_{\alpha - 1 - \delta}(t) \ \leq \ \frac{(\alpha - 1 - \delta) t \varepsilon(t)}{2} \| x^* \|^2$ starting from $t^{**}$. Repeating the same argument we derive that $t \mapsto E_{\alpha - 1 - \delta}(t)$ is bounded. The function $t \mapsto \| x(t) - x^* \|$ is also bounded and so is the trajectory $x$. Integrating \eqref{Energy_Pre_1} one may additionally obtain the integrability of $t \| \dot x(t) \|^2$. From the integrability of $\left( (\alpha - 3) t b(t) - t^2 \dot b(t) - \beta (\alpha - 2) \right) \left( \Phi_{\lambda(t)} (x(t)) - \Phi^* \right)$ and \eqref{A_6} we deduce
\[
\int_{t_0}^{+\infty} t b(t) \left( \Phi_{\lambda(s)} (x(s)) - \Phi^* \right) \ < \ +\infty.
\]

\end{proof}

The next theorem shows that we can actually improve the rates of convergence of the function values in case $\alpha > 3$.

\begin{theorem}\label{Th_1}

Assume that $\alpha > 3$ and \eqref{A_9}, \eqref{integrability_T}, \eqref{tbt} and \eqref{A_6} hold. Then
\begin{equation}\label{3}
    t \left\langle \left( b(t) - \frac{\beta}{t} \right) \nabla \Phi_{\lambda(t)} (x(t)), x(t) - x^* \right\rangle \in L^1 \big( [t_0, +\infty), \mathbb{R} \big).
\end{equation}
In addition, $\lim_{t \to +\infty} \psi(t) = 0$, where for $2 \leq q \leq \alpha - 1$
\[
\psi(t) = (t^2 b(t) - \beta (q + 2 - \alpha)t) \left( \Phi_{\lambda(t)} (x(t)) - \Phi^* \right) + \frac{t^2 \varepsilon(t)}{2} \| x(t) \|^2 + \frac{t^2}{2} \| \dot x(t) + \beta \nabla \Phi_{\lambda(t)} (x(t)) \|^2,
\]
which in particular means
\begin{equation}\label{Rates}
\begin{split}
    \Phi_{\lambda(t)} (x(t)) - \Phi^* \ &= \ o \left( \frac{1}{t^2 b(t)} \right) \text{ as } t \to +\infty, \\
    \| \dot x(t) + \beta \nabla \Phi_{\lambda(t)} (x(t)) \| \ &= \ o\left( \frac{1}{t} \right) \text{ as } t \to +\infty
\end{split}
\end{equation}
and moreover,
\[
\Phi(\prox\nolimits_{\lambda(t) \Phi}(x(t))) - \Phi^* \ = \ o \left( \frac{1}{t^2 b(t)} \right) \text{ as } t \to +\infty,
\]
\[
\| \prox\nolimits_{\lambda(t) \Phi} (x(t)) - x(t) \| \ = \ o \left( \frac{\sqrt{\lambda(t)}}{t \sqrt{b(t)}} \right) \text{ as } t \to +\infty
\]
and 
\[
\| \nabla \Phi_{\lambda(t)} (x(t)) \| \ = \ o \left( \frac{1}{t \sqrt{b(t) \lambda(t)}} \right) \text{ as } t \to +\infty.
\]

\end{theorem}

\begin{proof}

({\em i}) Let us first prove an auxiliary estimate \eqref{3}, which will allows us to obtain the rest of the desired results. We return to
\begin{align*}
    \dot E_q(t) \ \leq \ &\left( 2t b(t) + t^2 \dot b(t) - \beta (q + 2 - \alpha) \right) \left( \Phi_{\lambda(t)} (x(t)) - \Phi^* \right) + (q + 1 - \alpha) t \| \dot x(t) \|^2 \\
    &- \ \left( \left( t^2 b(t) - \beta (q + 2 - \alpha)t \right) \frac{\dot \lambda(t)}{2} - \beta^2 t + \beta t^2 \left( b(t) - \frac{1}{a} \right) \right) \| \nabla \Phi_{\lambda(t)} (x(t)) \|^2 \\
    &+ \ \frac{4 t \varepsilon(t) + 2 t^2 \dot \varepsilon(t) + a \beta t^2 \varepsilon^2(t)}{4} \| x(t) \|^2 \\
    &- \ q t \left\langle \left( b(t) - \frac{\beta}{t} \right) \nabla \Phi_{\lambda(t)} (x(t)) + \varepsilon(t) x(t), x(t) - x^* \right\rangle.
\end{align*}
Under condition \eqref{A_9} we deduce starting from $t^{**}$
\begin{align*}
    \dot E_q(t) \ \leq \ &\left( 2t b(t) + t^2 \dot b(t) - \beta (q + 2 - \alpha) \right) \left( \Phi_{\lambda(t)} (x(t)) - \Phi^* \right) \\
    &+ \ t \varepsilon(t) \| x(t) \|^2 - q t \left\langle \left( b(t) - \frac{\beta}{t} \right) \nabla \Phi_{\lambda(t)} (x(t)) + \varepsilon(t) x(t), x(t) - x^* \right\rangle.
\end{align*}
Integrating the last inequality on $[t_0, t]$ we obtain
\begin{equation}\label{5}
\begin{split}
    &\int_{t_0}^t q s \left\langle \left( b(s) - \frac{\beta}{s} \right) \nabla \Phi_{\lambda(s)} (x(s)), x(s) - x^* \right\rangle ds \ \leq \ E_q(t_0) - E_q(t) + \int_{t_0}^t s \varepsilon(s) \| x(s) \|^2 ds \\
    &+ \ \int_{t_0}^t \left( 2s b(s) + s^2 \dot b(s) - \beta (q + 2 - \alpha) \right) \left( \Phi_{\lambda(s)} (x(s)) - \Phi^* \right) ds - \int_{t_0}^t q s \langle \varepsilon(s) x(s), x(s) - x^* \rangle.
\end{split}
\end{equation}
Since the gradient $ \nabla \Phi_{\lambda} $ is monotone, we know that $\left\langle \nabla \Phi_{\lambda(t)} (x(t)), x(t) - x^* \right\rangle \geq 0$. Moreover,
\begin{equation}\label{A_5_0}
    -q t \langle \varepsilon(t) x(t), x(t) - x^* \rangle \ \leq \ \frac{q t \varepsilon(t)}{2} \left( \| x(t) \|^2 + \| x(t) - x^* \|^2 \right).
\end{equation}
Notice that by \eqref{A_6} we have 
\[
(\alpha - 3 - \delta) t b(t) - t^2 \dot b(t) + \beta (2 - \alpha) \ > \ 0
\]
or
\[
(\alpha - 3 - \delta) t b(t) - t^2 \dot b(t) \ > \ \beta (\alpha - 2).
\]
Obviuosly,
\[
(\alpha - 3 - \delta) t b(t) - t^2 \dot b(t) \ > \ \beta (\alpha - 2) \ > \ \beta (\alpha - 2 - q) \text{ for every } q \in (2, \alpha - 1).
\]
Introducing $\delta_1 = \alpha - 1 - \delta > 0$ (by the choice of $\delta$) we obtain
\[
(\delta_1 - 2) t b(t) - t^2 \dot b(t) \ > \ -\beta (q + 2 - \alpha)
\]
or
\[
2 t b(t) + t^2 \dot b(t) - \beta (q + 2 - \alpha) \ < \ \delta_1 t b(t).
\]
From Theorem \ref{Th_0} we know that $ t b(t) \left( \Phi_{\lambda(s)} (x(s)) - \Phi^* \right)$ is integrable and therefore so is \\
$\left( 2 t b(t) + t^2 \dot b(t) - \beta (q + 2 - \alpha) \right) \left( \Phi_{\lambda(s)} (x(s)) - \Phi^* \right)$. Since the function $t \mapsto E_q(t)$ is bounded and the rest of the right hand side of \eqref{5} belongs to $L^1 \big( [t_0, +\infty), \mathbb{R} \big)$ by Theorem \ref{Th_0} and \eqref{A_5_0}, we conclude with \eqref{3} due to \eqref{tbt}. 

({\em ii}) In order to derive the convergence rates for the quantities of our interest we require some additional results. Our nearest goal is to establish the existence of the limits 
\[ 
\lim_{t \to +\infty} \| x(t) - x^* \| \text{ and } \lim_{t \to +\infty} t \left\langle \dot x(t) + \beta \nabla \Phi_{\lambda(t)} (x(t)), x(t) - x^* \right\rangle.
\]  
Consider (as was done in \cite{BCL, BK}) for two different $q_1, q_2 \in (2, \alpha - 1)$ and for every $t \geq t_0$ the difference
\begin{align*}
    E_{q_1}(t) - E_{q_2}(t) \ = \ &(t^2 b(t) - \beta (q_1 + 2 - \alpha)t) \left( \Phi_{\lambda(t)} (x(t)) - \Phi^* \right) + \frac{t^2 \varepsilon(t)}{2} \| x(t) \|^2 \\
    &+ \ \frac{1}{2} \| q_1 (x(t) - x^*) + t (\dot x(t) + \beta \nabla \Phi_{\lambda(t)} (x(t)) \|^2 + \frac{q_1 (\alpha - 1 - q_1)}{2} \| x(t) - x^* \|^2 \\
    &- (t^2 b(t) - \beta (q_2 + 2 - \alpha)t) \left( \Phi_{\lambda(t)} (x(t)) - \Phi^* \right) - \frac{t^2 \varepsilon(t)}{2} \| x(t) \|^2 \\
    &- \ \frac{1}{2} \| q_2 (x(t) - x^*) + t (\dot x(t) + \beta \nabla \Phi_{\lambda(t)} (x(t)) \|^2 - \frac{q_2 (\alpha - 1 - q_2)}{2} \| x(t) - x^* \|^2 \\
    = \ & (q_1 - q_2) \Bigg( -\beta t \left( \Phi_{\lambda(t)} (x(t)) - \Phi^* \right) + t \left\langle \dot x(t) + \beta \nabla \Phi_{\lambda(t)} (x(t)), x(t) - x^* \right\rangle \\
    &+ \ \frac{\alpha - 1}{2} \| x(t) - x^* \|^2 \Bigg).
\end{align*}
As we have established earlier in Theorem \ref{Th_0} the limits of $E_{q_1}(t) - E_{q_2}(t)$ and $t \left( \Phi_{\lambda(t)} (x(t)) - \Phi^* \right)$ exists (the latter is actually zero). Therefore, the limit
\[
\lim_{t \to +\infty} \left( t \left\langle \dot x(t) + \beta \nabla \Phi_{\lambda(t)} (x(t)), x(t) - x^* \right\rangle + \frac{\alpha - 1}{2} \| x(t) - x^* \|^2 \right) \text{ also exists}.
\]
Let us introduce for every $t \geq t_0$ two auxiliary functions
\[
k(t) \ = \ t \left\langle \dot x(t) + \beta \nabla \Phi_{\lambda(t)} (x(t)), x(t) - x^* \right\rangle + \frac{\alpha - 1}{2} \| x(t) - x^* \|^2
\]
and
\[
r(t) \ = \ \frac{1}{2} \| x(t) - x^* \|^2 + \beta \int_{t_0}^t \left\langle \nabla \Phi_{\lambda(s)} (x(s)), x(s) - x^* \right\rangle ds.
\]
Noticing that
\[
\dot r(t) \ = \ \langle x(t) - x^*, \dot x(t) \rangle + \beta \left\langle \nabla \Phi_{\lambda(t)} (x(t)), x(t) - x^* \right\rangle
\]
we may write for every $t \geq t_0$
\[
(\alpha - 1) r(t) + t \dot r(t) = k(t) + \beta (\alpha - 1) \int_{t_0}^t \left\langle \nabla \Phi_{\lambda(s)} (x(s)), x(s) - x^* \right\rangle ds.
\]
From the fact that $\lim_{t \to +\infty} k(t)$ exists using \eqref{3} we obtain that $ \lim_{t \to +\infty} (\alpha - 1) r(t) + t \dot r(t) $ also exists. Applying Lemma \ref{L_2} we deduce the existence of the limit $\lim_{t \to +\infty} r(t)$. Using \eqref{3} again we obtain the existence of the limits $\lim_{t \to +\infty} \| x(t) - x^* \|$ and $\lim_{t \to +\infty} t \left\langle \dot x(t) + \beta \nabla \Phi_{\lambda(t)} (x(t)), x(t) - x^* \right\rangle$.

({\em iii}) Finally, we are in position to prove \eqref{Rates} and the rest of the convergence rates. The key idea is to show that the limit 
\[
\lim_{t \to +\infty} \left( (t^2 b(t) - \beta (q + 2 - \alpha)t) \left( \Phi_{\lambda(t)} (x(t)) - \Phi^* \right) + \frac{t^2 \varepsilon(t)}{2} \| x(t) \|^2 + \frac{t^2}{2} \| \dot x(t) + \beta \nabla \Phi_{\lambda(t)} (x(t)) \|^2 \right)
\]
exists and is actually zero. Let us return to the definition of our energy functional and rewrite it as
\begin{align*}
    E_q(t) \ = \ &(t^2 b(t) - \beta (q + 2 - \alpha)t) \left( \Phi_{\lambda(t)} (x(t)) - \Phi^* \right) + \frac{t^2 \varepsilon(t)}{2} \| x(t) \|^2 \\
    &+ \ \frac{t^2}{2} \| \dot x(t) + \beta \nabla \Phi_{\lambda(t)} (x(t) \|^2 + q t \left\langle \dot x(t) + \beta \nabla \Phi_{\lambda(t)} (x(t)), x(t) - x^* \right\rangle + \frac{q (\alpha - 1)}{2} \| x(t) - x^* \|^2.
\end{align*}
Since the limits 
\[
\lim_{t \to +\infty} E_q(t) \text{ and } \lim_{t \to +\infty} \left( q t \left\langle \dot x(t) + \beta \nabla \Phi_{\lambda(t)} (x(t), x(t) - x^* \right\rangle + \frac{q (\alpha - 1)}{2} \| x(t) - x^* \|^2 \right) \text{ exist,}
\]
it follows that
\[
\lim_{t \to +\infty} \left( (t^2 b(t) - \beta (q + 2 - \alpha)t) \left( \Phi_{\lambda(t)} (x(t)) - \Phi^* \right) + \frac{t^2 \varepsilon(t)}{2} \| x(t) \|^2 + \frac{t^2}{2} \| \dot x(t) + \beta \nabla \Phi_{\lambda(t)} (x(t)) \|^2 \right)
\]
exists as well. Denote 
\[
\psi(t) = (t^2 b(t) - \beta (q + 2 - \alpha)t) \left( \Phi_{\lambda(t)} (x(t)) - \Phi^* \right) + \frac{t^2 \varepsilon(t)}{2} \| x(t) \|^2 + \frac{t^2}{2} \| \dot x(t) + \beta \nabla \Phi_{\lambda(t)} (x(t)) \|^2
\] 
and consider
\begin{equation}\label{4}
    0 \ \leq \ \frac{\psi(t)}{t} \ \leq \ 2t b(t) \left( \Phi_{\lambda(t)} (x(t)) - \Phi^* \right) + \frac{t \varepsilon(t)}{2} \| x(t) \|^2 + \frac{t}{2} \| \dot x(t) + \beta \nabla \Phi_{\lambda(t)} (x(t)) \|^2.
\end{equation}
Let us show that the right hand side of \eqref{4} is integrable. Indeed, the first term is integrable by Theorem \ref{Th_0}. As we have also established in Theorem \ref{Th_0}, starting from $t^{**}$
\[
\Big( t b(t) - \beta \Big) \frac{\dot \lambda(t)}{2} + \beta t \left( b(t) - \frac{1}{a} \right) \ \geq \ \beta^2,
\]
where $a \geq 1$. Then, by \eqref{tbt} and $\dot \lambda(t) \geq 0$ for all $t \geq t_0$, we deduce that there exists $t_1 \geq t^{**}$ such that for all $t \geq t_1$
\[
\Big( t b(t) - \beta \Big) \frac{\dot \lambda(t)}{2} + \beta t \left( b(t) - \frac{1}{a} \right) \ \geq \ \frac{3 \beta^2}{2}
\]
or
\[
t^2 b(t) \left( \frac{\dot \lambda(t)}{2} + \beta \left( 1 - \frac{1}{a b(t)} \right) \right) \ \geq \ \left( 3 \beta + \dot \lambda(t) \right) \frac{\beta t}{2}
\]
or
\[
\left( t^2 b(t) - \beta t \right) \frac{\dot \lambda(t)}{2} - \beta^2 t + \beta t^2 \left( b(t) - \frac{1}{a} \right) \ \geq \ \frac{\beta^2 t}{2},
\]
So, by Theorem \ref{Th_0} the right hand side of \eqref{4} belongs to $L^1 \big( [t_1, +\infty), \mathbb{R} \big)$. Therefore, $\frac{\psi(t)}{t}$ also belongs to $L^1 \big( [t_1, +\infty), \mathbb{R} \big)$ and since the limit $\lim_{t \to +\infty} \psi(t)$ exists we deduce that it should be actually zero, which gives us \eqref{Rates}. To complete the proof notice that by the definition of the proximal mapping, we have
\begin{equation*}
\Phi_{\lambda(t)}(x(t)) - \Phi^* = \ \Phi(\prox\nolimits_{\lambda(t) \Phi}(x(t))) - \Phi^* + \frac{1}{2\lambda(t)} \| \prox\nolimits_{\lambda(t) \Phi} (x(t)) - x(t) \|^2 \quad \forall t \geq t_0.
\end{equation*}
The conclusion follows immediately from \eqref{Morprox} and \eqref{Rates}.

\end{proof}

\end{section}

\begin{section}{Strong convergence of the trajectories}

In this chapter we will establish the strong convergence of the trajectories to the minimal norm element of $\argmin \Phi$. 

In order to do so, we will need to modify assumption \eqref{integrability_T} from the previous chapter:

\begin{framed}
\begin{equation}\label{integrability_T_0}
    \int_{t_0}^{+\infty} \frac{\varepsilon(t)}{t b(t)} dt \ < \ +\infty.
\end{equation}    
\end{framed}

Before moving to the main point of the section, let us prove an auxiliary result first. 

\begin{theorem}\label{Str_conv_aux_th}

Suppose that $\alpha > 3$, the function $\lambda$ is bounded for all $t \geq t_0$ and \eqref{A_6_0}, \eqref{A_9}, \eqref{tbt} and \eqref{integrability_T_0} hold. Then
\[
\lim_{t \to +\infty} \| \prox\nolimits_{\lambda(t) \Phi} (x(t)) - x(t) \| \ = \ 0
\]
and
\[
\lim_{t \to +\infty} \Phi \left( \prox\nolimits_{\lambda(t) \Phi} (x(t)) \right) - \Phi^* \ = \ 0.
\]
\end{theorem}

\begin{proof}

Let us return to \eqref{Energy_FF_0}:
\[
\dot E_{\alpha - 1}(t) \ \leq \ (\alpha - 1) t \frac{\varepsilon(t)}{2} \| x^* \|^2.
\]
Let us integrate the last inequality on [T, t]
\[
E_{\alpha - 1}(t) \ \leq \ E_{\alpha - 1}(T) + \frac{(\alpha - 1) \| x^* \|^2}{2} \int_T^t s \varepsilon(s) ds.
\]
On the other hand, for every $t \geq t_0$
\[
E_{\alpha - 1}(t) \ \geq \ (t^2 b(t) - \beta t) \left( \Phi_{\lambda(t)} (x(t)) - \Phi^* \right).
\]
Thus,
\[
\Phi_{\lambda (t)}(x(t)) - \Phi^* \ \leq \ \frac{E_{\alpha - 1}(T)}{t^2 b(t) - \beta t} + \frac{(\alpha - 1) \| x^* \|^2}{2(t^2 b(t) - \beta t)} \int_T^t s \varepsilon(s) ds.
\]
We deduce due to \eqref{integrability_T_0} and the Lemma \ref{L_3} that
\[
\lim_{t \to +\infty} \frac{1}{t^2 b(t)} \int_T^t s^2 b(s) \frac{\varepsilon(s)}{s b(s)} ds \ = \ 0.
\]
Therefore,
\[
\lim_{t \to +\infty} \frac{(\alpha - 1) \| x^* \|^2}{2(t^2 b(t) - \beta t)} \int_T^t s \varepsilon(s) ds \ = \ 0
\]
and clearly
\[
\lim_{t \to +\infty} \frac{E_{\alpha - 1}(T)}{t^2 b(t) - \beta t} \ = \ 0.
\]
Thus, we establish 
\[
\lim_{t \to +\infty} \Phi_{\lambda(t)} (x(t)) - \Phi^* \ = \ 0.
\]
By the definition of the proximal mapping
\[
\Phi_{\lambda(t)}(x(t)) - \Phi^* = \ \Phi \left( \prox\nolimits_{\lambda(t) \Phi}(x(t)) \right) - \Phi^* + \frac{1}{2\lambda(t)} \| \prox\nolimits_{\lambda(t) \Phi} (x(t)) - x(t) \|^2 \quad \forall t \geq t_0.
\]
Using the fact that $\lambda$ is bounded for all $t \geq t_0$ we deduce
\[
\lim_{t \to +\infty} \| \prox\nolimits_{\lambda(t) \Phi} (x(t)) - x(t) \| \ = \ 0
\]
and
\[
\lim_{t \to +\infty} \Phi \left( \prox\nolimits_{\lambda(t) \Phi} (x(t)) \right) - \Phi^* \ = \ 0.
\]

\end{proof}

For the remaining part of this section we will use a different energy functional. Inspired by \cite{BCL} we introduce the following functional, which we will heavily rely on throughout this section
\begin{equation}\label{Energy}
    E_{p, q}(t) = t^{p+1} \left( t b(t) + \beta (\alpha - p - q - 2) \right) \left( \Phi_{\lambda(t)} (x(t)) - \Phi^* \right) + \frac{\varepsilon(t) t^{p+2}}{2} \left( \| x(t) \|^2 - \| x^* \|^2 \right) + \frac{t^p}{2} \left \| v(t) \right \|^2,
\end{equation}
where $v(t) = q (x(t) - x^*) + t (\dot x(t) + \beta \nabla \Phi_{\lambda(t)} (x(t)) $ and $p, \ q \geq 0$.

The proof of the following theorem draws inspiration from \cite{AC_1, ACR_0, BCL}.

\begin{theorem}\label{Str_conv}

Suppose that $\lambda$ is bounded for all $t \geq t_0$, $\alpha > 3$, $b(t_0) \geq \frac{1}{2} + \frac{\beta}{t_0}$ and \eqref{A_6_0}, \eqref{A_9} and \eqref{integrability_T_0} are fulfilled. Suppose additionally that for all $t \geq t_0$
\begin{equation}\label{A_10}
    \left( \frac{\alpha}{3} - 1 \right) t b(t) - t^2 \dot b(t) + \frac{\alpha \beta}{3} \ \geq \ 0
\end{equation}
and moreover that for all $t \geq t_0$
\begin{equation}\label{A_2}
    2 \alpha (\alpha - 3) - 9 t^2 \varepsilon(t) + 6 \alpha \beta \ \leq \ 0,
\end{equation}
\begin{equation}\label{A_3}
    18 \beta t + 9 \beta \dot \lambda(t) - 9 t b(t) \left( \dot \lambda(t) + 2 \beta \right) + 3(\alpha + 3) \beta^2 + \alpha^2 \beta  \ \leq \ 0
\end{equation}
and
\begin{equation}\label{A_5}
    \lim_{t \to +\infty} \frac{\beta}{t^{\frac{\alpha}{3} + 1} \varepsilon(t)} \int_{t_0}^t s^{\frac{\alpha}{3} + 1} \varepsilon^2(s) ds \ = \ 0.
\end{equation}
If $x : [t_0,+\infty) \mapsto H$ is a solution to \eqref{Syst} and the trajectory $x(t)$ stays either inside or outside the ball $B(0, \| x^* \|)$, then $x(t)$ converges to minimal norm solution $x^* = \proj_{\argmin \Phi}(0)$, as $t \to +\infty$. Otherwise, $\liminf_{t \to +\infty} \| x(t) - x^* \| = 0$.

\end{theorem}

\begin{proof}

As in \cite{BCL} we will consider several cases with respect to the trajectory $x$ staying either inside or outside the ball $B \left( 0, \| x^* \| \right)$.

\begin{center}
    Case I.
\end{center}

Assume that the trajectory $x$ stays in the complement of the ball $B$ for all $t \geq t_0$. This means nothing but $\| x(t) \| \geq \| x^* \|$ for every $t \geq t_0$. 

{(\em i)} Our nearest goal is to obtain the upper bound for the derivative of $E_{p, q}$. In order to do so, let us evaluate its time derivative for every $t \geq t_0$ first.
\begin{equation}\label{2}
\begin{split}
    \frac{d}{dt} E_{p, q}(t) \ = \ &t^p \left( (p + 2) t b(t) + t^2 \dot b(t) + (p + 1) \beta (\alpha - p - q - 2) \right) \left( \Phi_{\lambda(t)} (x(t)) - \Phi^* \right) \\
    &+ \ t^{p+1} \left( t b(t) + \beta (\alpha - p - q - 2) \right) \left( \left\langle \nabla \Phi_{\lambda(t)} (x(t)), \dot x(t) \right\rangle - \frac{\dot \lambda(t)}{2} \| \nabla \Phi_{\lambda(t)} (x(t)) \|^2 \right) \\
    &+ \ \frac{(p+2) t^{p+1} \varepsilon(t) + t^{p+2} \dot \varepsilon(t)}{2} \left( \| x(t) \|^2 - \| x^* \|^2 \right) + t^{p+2} \varepsilon(t) \langle \dot x(t), x(t) \rangle \\
    &+ \ \frac{p t^{p-1}}{2} \| q(x(t) - x^*) + t ( \dot x(t) + \beta \nabla \Phi_{\lambda(t)} (x(t)) \|^2 + t^p \langle \dot v(t), v(t) \rangle.
\end{split}
\end{equation}
Consider for every $t \geq t_0$ the inner product $\langle \dot v(t), v(t) \rangle$:
\begin{align*}
    &\Bigg\langle (q + 1) \dot x(t) + \beta \nabla \Phi_{\lambda(t)} (x(t)) + t \left( \ddot x(t) + \beta \frac{d}{dt} \nabla \Phi_{\lambda(t)} (x(t)) \right), q(x(t) - x^*) + t ( \dot x(t) \\
    &+ \ \beta \nabla \Phi_{\lambda(t)} (x(t)) \Bigg\rangle \\
    = \ &\Bigg\langle (q + 1 - \alpha) \dot x(t) + \beta \nabla \Phi_{\lambda(t)} (x(t)) - t \Big( b(t) \nabla \Phi_{\lambda(t)} (x(t)) + \varepsilon(t) x(t) \Big), q(x(t) - x^*) + t ( \dot x(t) \\
    &+ \ \beta \nabla \Phi_{\lambda(t)} (x(t)) \Bigg\rangle \\
    = \ &q (q + 1 - \alpha) \langle \dot x(t), x(t) - x^* \rangle + (q + 1 - \alpha) t \left( \| \dot x(t) \|^2 + \left\langle \beta \nabla \Phi_{\lambda(t)} (x(t)),\dot x(t) \right\rangle \right) \\
    &+ \ \beta q \langle \nabla \Phi_{\lambda(t)} (x(t)), x(t) - x^* \rangle + \beta t \langle \nabla \Phi_{\lambda(t)} (x(t)), \dot x(t) \rangle + \beta^2 t \| \nabla \Phi_{\lambda(t)} (x(t)) \|^2 \\
    &- \ q t \left\langle b(t) \nabla \Phi_{\lambda(t)} (x(t)) + \varepsilon(t) x(t), x(t) - x^* \right\rangle - t^2 \left\langle b(t) \nabla \Phi_{\lambda(t)} (x(t)) + \varepsilon(t) x(t), \dot x(t) \right\rangle \\
    &- \ \beta t^2 \big\langle b(t) \nabla \Phi_{\lambda(t)} (x(t)) + \varepsilon(t) x(t), \nabla \Phi_{\lambda(t)} (x(t)) \big\rangle,
\end{align*}
where above we used \eqref{Syst}. Consider now for every $t \geq t_0$,
\begin{align*}
    &\| q(x(t) - x^*) + t ( \dot x(t) + \beta \nabla \Phi_{\lambda(t)} (x(t)) ) \|^2 \ = \ q^2 \| x(t) - x^* \|^2 + 2q t \langle \dot x(t), x(t) - x^* \rangle \\
    &+ \ 2q \beta t \langle \nabla \Phi_{\lambda(t)} (x(t)), x(t) - x^* \rangle + t^2 \| \dot x(t) \|^2 + 2 \beta t^2 \langle \nabla \Phi_{\lambda(t)} (x(t)), \dot x(t) \rangle + \beta^2 t^2 \| \nabla \Phi_{\lambda(t)} (x(t)) \|^2.
\end{align*}
The two estimates that we made above lead to \eqref{2} becoming
\begin{align*}
    \frac{d}{dt} E_{p, q}(t) \ = \ &t^p \left( (p + 2) t b(t) + t^2 \dot b(t) + (p + 1) \beta (\alpha - p - q - 2) \right) \left( \Phi_{\lambda(t)} (x(t)) - \Phi^* \right) \\
    &+ \ \frac{(p + 2) t^{p+1} \varepsilon(t) + t^{p+2} \dot \varepsilon(t)}{2} \left( \| x(t) \|^2 - \| x^* \|^2 \right) + \frac{p q^2 t^{p-1}}{2} \| x(t) - x^* \|^2 \\
    &+ \ \frac{(p + 2) \beta^2 t^{p+1}}{2} \| \nabla \Phi_{\lambda(t)} (x(t)) \|^2 + \left( q + 1 - \alpha + \frac{p}{2} \right) t^{p+1} \| \dot x(t) \|^2 \\
    &+ \ q (q + 1 - \alpha + p) t^p \big\langle \dot x(t), x(t) - x^* \big\rangle + q \beta (p + 1) t^p \big\langle \nabla \Phi_{\lambda(t)} (x(t)), x(t) - x^* \big\rangle \\
    &- \ q t^{p+1} \Big\langle b(t) \nabla \Phi_{\lambda(t)} (x(t)) + \varepsilon(t) x(t), x(t) - x^* \Big\rangle \\
    &- \ \beta t^{p+2} \Big\langle b(t) \nabla \Phi_{\lambda(t)} (x(t)) + \varepsilon(t) x(t), \nabla \Phi_{\lambda(t)} (x(t)) \Big\rangle \\
    &- \ \frac{\dot \lambda(t) t^{p+1} \left( t b(t) + \beta (\alpha - p - q - 2) \right)}{2} \| \nabla \Phi_{\lambda(t)} (x(t)) \|^2.
\end{align*}
Let us apply the gradient inequality to the strongly convex function $x \mapsto b(t) \Phi_{\lambda(t)} (x) + \frac{\varepsilon(t) \| x \|^2}{2} $:
\begin{align*}
    &-\Big\langle b(t) \nabla \Phi_{\lambda(t)} (x(t)) + \varepsilon(t) x(t), x(t) - x^* \Big\rangle + \frac{\varepsilon(t) \| x(t) - x^* \|^2}{2} \\
    \leq \ &\left( b(t) \Phi^* + \frac{\varepsilon(t) \| x^* \|^2}{2} \right) - \left( b(t) \Phi_{\lambda(t)} (x(t)) + \frac{\varepsilon(t) \| x(t) \|^2}{2} \right)
\end{align*}
and thus
\begin{align*}
    &-q t^{p+1} \Big\langle b(t) \nabla \Phi_{\lambda(t)} (x(t)) + \varepsilon(t) x(t), x(t) - x^* \Big\rangle \ \leq \ -q t^{p+1} b(t) \left( \Phi_{\lambda(t)} (x(t)) - \Phi^* \right) \\
    &- q t^{p+1}\frac{\varepsilon(t)}{2} \left( \| x(t) \|^2 - \| x^* \|^2 \right) - q t^{p+1} \frac{\varepsilon(t) \| x(t)-x^* \|^2}{2}
\end{align*}
for every $t \geq t_0$. So, noticing that
\begin{align*}
    &-\beta t^{p+2} \Big\langle b(t) \nabla \Phi_{\lambda(t)} (x(t)) + \varepsilon(t) x(t), \nabla \Phi_{\lambda(t)} (x(t)) \Big\rangle \\
    = \ &-\beta t^{p+2} b(t) \| \nabla \Phi_{\lambda(t)} (x(t)) \|^2 - \beta t^{p+2} \varepsilon(t) \Big\langle x(t), \nabla \Phi_{\lambda(t)} (x(t)) \Big\rangle
\end{align*}
we deduce
\begin{align*}
    \frac{d}{dt} E_{p, q}(t) \ \leq \ &t^p \left( (p + 2 - q) t b(t) + t^2 \dot b(t) + (p + 1) \beta (\alpha - p - q - 2) \right) \left( \Phi_{\lambda(t)} (x(t)) - \Phi^* \right) \\
    &+ \ \frac{(p + 2 - q) t^{p+1} \varepsilon(t) + t^{p+2} \dot \varepsilon(t)}{2} \left( \| x(t) \|^2 - \| x^* \|^2 \right) \\
    &+ \ \left( \frac{p q^2 t^{p-1}}{2} - \frac{q t^{p+1} \varepsilon(t)}{2} \right) \| x(t) - x^* \|^2 \\
    &+ \ \frac{(p+2) \beta^2 t^{p+1} - 2 \beta t^{p+2} b(t) - \dot \lambda(t) t^{p+1} \left( t b(t) + \beta (\alpha - p - q - 2) \right)}{2} \| \nabla \Phi_{\lambda(t)} (x(t)) \|^2 \\
    &+ \ \left( q + 1 - \alpha + \frac{p}{2} \right) t^{p+1} \| \dot x(t) \|^2 + q (q + 1 - \alpha + p) t^p \big\langle \dot x(t), x(t) - x^* \big\rangle \\
    &+ \ q \beta (p + 1) t^p \big\langle \nabla \Phi_{\lambda(t)} (x(t)), x(t) - x^* \big\rangle - \beta t^{p+2} \varepsilon(t) \Big\langle x(t), \nabla \Phi_{\lambda(t)} (x(t)) \Big\rangle.
\end{align*}
In order to proceed further we will need the following estimates:
\begin{align*}
    q \beta (p + 1) t^p \big\langle \nabla \Phi_{\lambda(t)} (x(t)), x(t) - x^* \big\rangle \ \leq \ \frac{q \beta (p + 1) t^{p+1}}{4 c^2} \| \nabla \Phi_{\lambda(t)} (x(t)) \|^2 + q \beta (p + 1) c^2 t^{p-1} \| x(t) - x^* \|^2
\end{align*}
and
\begin{align*}
    -\beta t^{p+2} \varepsilon(t) \Big\langle x(t), \nabla \Phi_{\lambda(t)} (x(t)) \Big\rangle \ \leq \ \frac{\beta t^{p+2}}{a} \| \nabla \Phi_{\lambda(t)} (x(t)) \|^2 + \frac{a \beta t^{p+2} \varepsilon^2(t)}{4} \| x(t) \|^2
\end{align*}
for every $t \geq t_0$, some $c \geq 1$ and $a \geq 1$. Thus,
\begin{align*}
    \frac{d}{dt} E_{p, q}(t) \ \leq \ &t^p \left( (p + 2 - q) t b(t) + t^2 \dot b(t) + (p + 1) \beta (\alpha - p - q - 2) \right) \left( \Phi_{\lambda(t)} (x(t)) - \Phi^* \right) \\
    &+ \ \left( \frac{(p + 2 - q) t^{p+1} \varepsilon(t) + t^{p+2} \dot \varepsilon(t)}{2} + \frac{a \beta t^{p+2} \varepsilon^2(t)}{4} \right) \| x(t) \|^2 \\
    &+ \ \left( \frac{p q^2 t^{p-1}}{2} - \frac{q t^{p+1} \varepsilon(t)}{2} + q \beta (p + 1) c^2 t^{p-1} \right) \| x(t) - x^* \|^2 \\
    &+ \ \Bigg( \frac{(p+2) \beta^2 t^{p+1} - 2 \beta t^{p+2} b(t) - \dot \lambda(t) t^{p+1} \left( t b(t) + \beta (\alpha - p - q - 2) \right)}{2} + \frac{q \beta (p + 1) t^{p+1}}{4 c^2} \\
    &+ \ \frac{\beta t^{p+2}}{a} \Bigg) \cdot \| \nabla \Phi_{\lambda(t)} (x(t)) \|^2 \\
    &+ \ \left( q + 1 - \alpha + \frac{p}{2} \right) t^{p+1} \| \dot x(t) \|^2 + q (q + 1 - \alpha + p) t^p \big\langle \dot x(t), x(t) - x^* \big\rangle \\
    &- \ \left( \frac{(p + 2 - q) t^{p+1} \varepsilon(t) + t^{p+2} \dot \varepsilon(t)}{2} \right) \| x^* \|^2.
\end{align*}
Let us fix
\[
q = \frac{2 \alpha}{3} \text{ and } p = \frac{\alpha - 3}{3}.
\]
First of all, due to this choice
\[
q + 1 - \alpha + p \ = \ 0 
\]
and thus we get rid of the term $\big\langle \dot x(t), x(t) - x^* \big\rangle$. Secondly,
\begin{equation}\label{*}
    q + 1 - \alpha + \frac{p}{2} \ = \ -\frac{p}{2} \ \leq \ 0.
\end{equation}
Then
\begin{equation}\label{**}
    p + 2 - q \ = \ 1 - \frac{\alpha}{3} \ \leq \ 0.
\end{equation}
So,
\begin{align*}
    \frac{d}{dt} E_{p, q}(t) \ \leq \ &t^p \left( (p + 2 - q) t b(t) + t^2 \dot b(t) + (p + 1) \beta (\alpha - p - q - 2) \right) \left( \Phi_{\lambda(t)} (x(t)) - \Phi^* \right) \\
    &+ \ \left( \frac{(p + 2 - q) t^{p+1} \varepsilon(t) + t^{p+2} \dot \varepsilon(t)}{2} + \frac{a \beta t^{p+2} \varepsilon^2(t)}{4} \right) \| x(t) \|^2 \\
    &+ \ \left( \frac{p q^2 t^{p-1}}{2} - \frac{q t^{p+1} \varepsilon(t)}{2} + q \beta (p + 1) c^2 t^{p-1} \right) \| x(t) - x^* \|^2 \\
    &+ \ \Bigg( \frac{(p+2) \beta^2 t^{p+1} - 2 \beta t^{p+2} b(t) - \dot \lambda(t) t^{p+1} \left( t b(t) + \beta (\alpha - p - q - 2) \right)}{2} + \frac{q \beta (p + 1) t^{p+1}}{4 c^2} \\
    &+ \ \frac{\beta t^{p+2}}{a} \Bigg) \cdot \| \nabla \Phi_{\lambda(t)} (x(t)) \|^2 \\
    &+ \ \left( q + 1 - \alpha + \frac{p}{2} \right) t^{p+1} \| \dot x(t) \|^2 - \left( \frac{(p + 2 - q) t^{p+1} \varepsilon(t) + t^{p+2} \dot \varepsilon(t)}{2} \right) \| x^* \|^2.
\end{align*}
Obviously, for $t$ large enough, say, $t \geq t_2 \geq t_0$ the following expression is non-positive due to \eqref{A_10} and $p + 1 = \frac{\alpha}{3} > 0$ and $\alpha - p - q - 2 = -1$
\[
(p + 2 - q) t b(t) + t^2 \dot b(t) + (p + 1) \beta (\alpha - p - q - 2) \ = \ \left( 1 - \frac{\alpha}{3} \right) t b(t) + t^2 \dot b(t) - \frac{\alpha \beta}{3} \ \leq \ 0.
\]
Moreover, from \eqref{A_2} it follows that for $c = 1$
\[
\frac{p q^2 t^{p-1}}{2} - \frac{q t^{p+1} \varepsilon(t)}{2} + q \beta (p + 1) c^2 t^{p-1} \ = \ \frac{\alpha t^{\frac{\alpha - 6}{3}}}{27} \left( 2 \alpha (\alpha - 3) - 9 t^2 \varepsilon(t) + 6 \alpha \beta \right) \ \leq \ 0
\]
for all $t \geq t_0$. Furthermore,
\begin{align*}
    &\left( \frac{(p + 2 - q) t^{p+1} \varepsilon(t) + t^{p+2} \dot \varepsilon(t)}{2} + \frac{a \beta t^{p+2} \varepsilon^2(t)}{4} \right) \| x(t) \|^2 - \left( \frac{(p + 2 - q) t^{p+1} \varepsilon(t) + t^{p+2} \dot \varepsilon(t)}{2} \right) \| x^* \|^2 \ = \\
    &\left( \frac{(p + 2 - q) t^{p+1} \varepsilon(t) + t^{p+2} \dot \varepsilon(t)}{2} + \frac{a \beta t^{p+2} \varepsilon^2(t)}{4} \right) \left( \| x(t) \|^2 - \| x^* \|^2 \right) + \frac{a \beta t^{p+2} \varepsilon^2(t)}{4} \| x^* \|^2.
\end{align*}
So, under the assumption \eqref{A_9} and the fact that $\| x(t) \| \geq \| x^* \|$ for all $t \geq t_0$ we deduce due to \eqref{**}
\[
\left( \frac{(p + 2 - q) t^{p+1} \varepsilon(t) + t^{p+2} \dot \varepsilon(t)}{2} + \frac{a \beta t^{p+2} \varepsilon^2(t)}{4} \right) \left( \| x(t) \|^2 - \| x^* \|^2 \right) \ \leq \ 0.
\]
Thus, under the assumptions \eqref{A_9}, \eqref{A_10}, \eqref{A_2} and \eqref{A_3} (the latest leads to the non-positivity of the coefficient of $\| \nabla \Phi_{\lambda}(x) \|^2$) we conclude due to \eqref{*} that for every $t \geq t_2$
\begin{equation}\label{Energy_FF}
    \frac{d}{dt} E_{p, q}(t) \ \leq \ \frac{a \beta t^{\frac{\alpha}{3} + 1} \varepsilon^2(t)}{4} \| x^* \|^2.
\end{equation}
{\em (ii)} Let us obtain now the lower bound for $E_{p ,q}$. Notice that for $p = \frac{\alpha - 3}{3}$ and $q = \frac{2 \alpha}{3}$ we have $\alpha - p - q = 1$ and
\begin{equation}\label{0}
\begin{split}
    E_{p, q}(t) \ &\geq \ t^{p+1} \left( t b(t) + \beta (\alpha - p - q - 2) \right) \left( \Phi_{\lambda(t)} (x(t)) - \Phi^* \right) + \frac{\varepsilon(t) t^{p+2}}{2} \left( \| x(t) \|^2 - \| x^* \|^2 \right) \\
    &= \ t^{p+1} \left( t b(t) - \beta \right) \left( \Phi_{\lambda(t)} (x(t)) - \Phi^* \right) + \frac{\varepsilon(t) t^{p+2}}{2} \left( \| x(t) \|^2 - \| x^* \|^2 \right) \\
    &\geq \ \frac{t^{p+2}}{2} \left( \Phi_{\lambda(t)} (x(t)) - \Phi^* \right) + \frac{\varepsilon(t) t^{p+2}}{2} \left( \| x(t) \|^2 - \| x^* \|^2 \right),
\end{split}
\end{equation}
since $t b(t) - \beta \geq \frac{t}{2}$ for every $t \geq t_0$ by $b(t_0) \geq \frac{1}{2} + \frac{\beta}{t_0}$ and $b$ being non-decreasing. On the other hand, applying the gradient inequality to the strongly convex function $\varphi_{\varepsilon(t), \lambda(t)}(x) \ = \ \frac{\Phi_{\lambda(t)}(x)}{2} + \frac{\varepsilon(t)}{2} \| x \|^2$ we deduce for $x_{\varepsilon(t), \lambda(t)} = \argmin_{H} \varphi_{\varepsilon(t), \lambda(t)}(x)$
\[
\varphi_{\varepsilon(t), \lambda(t)} (x) - \varphi_{\varepsilon(t), \lambda(t)} (x_{\varepsilon(t), \lambda(t)}) \ \geq \ \frac{\varepsilon(t)}{2} \| x - x_{\varepsilon(t), \lambda(t)} \|^2 \text{ for every } x \in H.
\]
By the definition of $\varphi_{\varepsilon(t), \lambda(t)} (x)$ we deduce
\begin{align*}
    \varphi_{\varepsilon(t), \lambda(t)} (x_{\varepsilon(t), \lambda(t)}) - \varphi_{\varepsilon(t), \lambda(t)} (x^*) \ &= \ \frac{1}{2} \left( \Phi_{\lambda(t)} (x_{\varepsilon(t), \lambda(t)}) - \Phi^* \right) + \frac{\varepsilon(t)}{2} \left( \| x_{\varepsilon(t), \lambda(t)} \|^2 - \| x^* \|^2 \right) \\
    &\geq \ \frac{\varepsilon(t)}{2} \left( \| x_{\varepsilon(t), \lambda(t)} \|^2 - \| x^* \|^2 \right).
\end{align*}
We may now add the last two inequalities to obtain
\begin{equation}\label{1}
    \varphi_{\varepsilon(t), \lambda(t)} (x) - \varphi_{\varepsilon(t), \lambda(t)} (x^*) \ \geq \ \frac{\varepsilon(t)}{2} \left( \| x - x_{\varepsilon(t), \lambda(t)} \|^2 + \| x_{\varepsilon(t), \lambda(t)} \|^2 - \| x^* \|^2 \right) \text{ for every } x \in H.
\end{equation}
Plugging \eqref{1} into \eqref{0} we conclude that for every $t \geq t_2$
\begin{equation}\label{LBE}
    E_{p, q}(t) \ \geq \ \frac{t^{p+2} \varepsilon(t)}{2} \left( \| x(t) - x_{\varepsilon(t), \lambda(t)} \|^2 + \| x_{\varepsilon(t), \lambda(t)} \|^2 - \| x^* \|^2 \right).
\end{equation}
{\em (iii)} Finally, using the lower and upper bounds for $E_{p, q}$ we can prove the strong convergence of the trajectories to a minimal norm solution. Integrating \eqref{Energy_FF} on $[t_2, t]$ we obtain
\[
E_{p, q}(t) \ \leq \ E_{p, q}(t_2)  + \frac{a \beta \| x^* \|^2}{4} \int_{t_2}^t s^{\frac{\alpha}{3} + 1} \varepsilon^2(s) ds
\]
and using \eqref{LBE} we deduce for every $t \geq t_2$
\[
\| x(t) - x_{\varepsilon(t), \lambda(t)} \|^2 \ \leq \ \| x^* \|^2 - \| x_{\varepsilon(t), \lambda(t)} \|^2 + \frac{2 E_{p, q} (t_2)}{t^{\frac{\alpha}{3} + 1} \varepsilon(t)} + \frac{a \beta \| x^* \|^2}{2 t^{\frac{\alpha}{3} + 1} \varepsilon(t)} \int_{t_2}^t s^{\frac{\alpha}{3} + 1} \varepsilon^2(s) ds.
\]
Note that due to \eqref{A_2}
\[
t^2 \varepsilon(t) \ \geq \ \frac{2 \alpha (\alpha - 3) + 6 \alpha \beta}{9} \ = \ \hat{C} \ \geq \ 0
\]
and
\[
t^{\frac{\alpha}{3} + 1} \varepsilon(t) \ = \ t^2 \varepsilon(t) t^{\frac{\alpha}{3} - 1} \ \geq \ \hat{C} t^{\frac{\alpha}{3} - 1}.
\]
Since $\alpha > 3$ we deduce
\[
\lim_{t \to +\infty} t^{\frac{\alpha}{3} + 1} \varepsilon(t) = +\infty
\]
and thus
\[
\lim_{t \to +\infty} \frac{2 E_{p, q} (t_2)}{t^{\frac{\alpha}{3} + 1} \varepsilon(t)} \ = \ 0.
\]
Finally, by \eqref{T_1} and \eqref{A_5} we conclude
\[
\lim_{t \to +\infty} x(t) = x^*.
\]

\begin{center}
    Case II.
\end{center}

Assume now the opposite to the first case, namely, $\| x(t) \| < \| x^* \|$ for every $t \geq t_0$. According to Theorem \ref{Str_conv_aux_th}
\[
\lim_{t \to +\infty} \| \prox\nolimits_{\lambda(t) \Phi} (x(t)) - x(t) \| \ = \ 0
\]
and
\[
\lim_{t \to +\infty} \Phi \left( \prox\nolimits_{\lambda(t) \Phi} (x(t)) \right) - \Phi^* \ = \ 0.
\]
Denote $\xi(t) = \prox\nolimits_{\lambda(t) \Phi}(x(t))$. Considering a sequence $ \{ t_k \}_{k \in \mathbb N} $ such that $ \{x(t_k)\}_{k \in \mathbb N} $ converges weakly to an element $ \hat x \in H $  as $ k \to \infty $, we notice that  $\{\xi(t_k)\}_{k \in \mathbb N} $ converges weakly to $\hat x$ as $ k \to \infty $. Now, the function $ \Phi $ being convex and lower semicontinuous in the weak topology, allows us to write
\[
\Phi(\hat x) \ \leq \ \liminf_{k \to \infty} \Phi(\xi(t_k)) \ = \ \lim_{t \to +\infty} \Phi(\xi(t)) \ = \ \Phi^* 
\]
and hence, $ \hat x \in \argmin \Phi$. The norm is weakly semicontinuous, so
\[
\| \hat x \| \ \leq \ \liminf_{k \to \infty} \| \xi(t_k) \| \ \leq \ \| x^* \|,
\]
which means that $\hat x = x^*$ by the uniqueness of the element of the minimum norm in $\argmin \Phi_{\lambda}$. Therefore, the trajectory $x$ converges weakly to $x^*$ and
\[
\| x^* \| \ \leq \ \liminf_{t \to +\infty} \| x(t) \| \ \leq \ \limsup_{t \to +\infty} \| x(t) \| \ \leq \ \| x^* \|
\]
and thus
\[
\lim_{t \to +\infty} \| x(t) \| \ = \ \| x^* \|.
\]
From this and the weak convergence of the trajectory $x$ follows the strong one: $\lim_{t \to +\infty} x(t) = x^*$.

\begin{center}
    Case III.
\end{center}
Assume that for $t \geq t_0$ the trajectory $x$ finds itself both inside and outside the ball $B(0, \| x^* \|)$. Since $x$ is continuous, there exists a sequence $\{ t_n \}_{n \in \mathbb{N}} \subseteq [t_0, +\infty)$ such that $t_n \to \infty$ as $n \to \infty$ and $\| x(t_n) \| = \| x^* \|$ for every $n \in \mathbb{N}$. Consider again a weak sequential cluster point $\hat x$ of the sequence $\{ x(t_n) \}_{n \in \mathbb{N}}$. By repeating the same argument as in the previous case we deduce the weak convergence of $\{ x(t_n) \}_{n \in \mathbb{N}}$ to $x^*$, as $n \to \infty$. Since $\| x(t_n) \| \to \| x^* \|$, as $n \to \infty$, we obtain that $\| x(t_n) - x^* \| \to 0$, as $n \to \infty$, which means $\liminf_{t \to +\infty} \| x(t) - x^* \| = 0$.

\end{proof}

\begin{remark}

In this section the condition $\dot b(t) \geq 0$ for all $t \geq t_0$ is not necessary. Our conjecture is that we can weaken the setting by omitting this condition and thus widen the range for $b$, including the functions that decay not faster than $\frac{1}{t^2}$ for the polynomial choice of parameters.

\end{remark}

\begin{remark}

There is no setting which guarantees both fast rates for the values and strong convergence of the trajectories. One of the future goal would be to develop a new approach (based on \cite{ABCR_0}), which would help us deduce these two results simultaneously.

\end{remark}

\end{section}

\begin{subsection}{Strong convergence of the trajectories in case $\alpha = 3$.}

Throughout this section we no longer require that $b$ is non-decreasing. In this case the analogue of Theorem \ref{Str_conv_aux_th} looks as follows

\begin{theorem}\label{Str_conv_aux_th_0}

Suppose that for all $t \geq t_0$ the function $\lambda$ is bounded, $b(t) \equiv b > 0$ is a constant function and \eqref{A_9} and \eqref{tbt} hold. Suppose additionally that \eqref{integrability_T_0} holds for constant $b$, namely
\begin{equation*}
    \int_{t_0}^{+\infty} \frac{\varepsilon(t)}{t} dt \ < \ +\infty.
\end{equation*}
Then
\[
\lim_{t \to +\infty} \| \prox\nolimits_{\lambda(t) \Phi} (x(t)) - x(t) \| \ = \ 0
\]
and
\[
\lim_{t \to +\infty} \Phi \left( \prox\nolimits_{\lambda(t) \Phi} (x(t)) \right) - \Phi^* \ = \ 0.
\]
\end{theorem}

\begin{proof}

In this case the energy functional becomes 
\[
    E_2(t) \ = \ (b t^2 - \beta t) \left( \Phi_{\lambda(t)} (x(t)) - \Phi^* \right) + \frac{t^2 \varepsilon(t)}{2} \| x(t) \|^2 + \frac{1}{2} \| 2 (x(t) - x^*) + t (\dot x(t) + \beta \nabla \Phi_{\lambda(t)} (x(t)) \|^2.
\]
Relation \eqref{Energy_Pre} thus becomes for all $t \geq t_0$
\begin{align*}
    \dot E_2(t) \ \leq \ &\beta \left( \Phi_{\lambda(t)} (x(t)) - \Phi^* \right) - \left( \Big( b t^2 - \beta t \Big) \frac{\dot \lambda(t)}{2} - \beta^2 t + \beta t^2 \left( b - \frac{1}{a} \right) \right) \| \nabla \Phi_{\lambda(t)} (x(t)) \|^2 \\
    &+ \ \frac{2 t^2 \dot \varepsilon(t) + a \beta t^2 \varepsilon^2(t)}{4} \| x(t) \|^2 - t \varepsilon(t) \| x(t) - x^* \|^2 + t \varepsilon(t) \| x^* \|^2.
\end{align*}
Thus, repeating the same arguments as in Theorem \ref{Th_0} we obtain
\[
\dot E_2(t) \ \leq \ \beta \left( \Phi_{\lambda(t)} (x(t)) - \Phi^* \right) + t \varepsilon(t) \| x^* \|^2.
\]
Let us multiply this expression with $t (b t - \beta)$ to obtain
\[
t (b t - \beta) \dot E_2(t) \ \leq \ \beta t (b t - \beta) \left( \Phi_{\lambda(t)} (x(t)) - \Phi^* \right) + t^2 (b t - \beta) \varepsilon(t) \| x^* \|^2 \ \leq \ \beta E_2(t) + t^2 (b t - \beta) \varepsilon(t) \| x^* \|^2.
\]
Now, we will divide by $(b t - \beta)^2$ to conclude
\[
\frac{t}{(b t - \beta)} \dot E_2(t) \ \leq \ \frac{\beta}{(b t - \beta)^2} E_2(t) + \frac{t^2}{(b t - \beta)} \varepsilon(t) \| x^* \|^2
\]
or
\[
\frac{d}{dt} \left( \frac{t}{b t - \beta} E_2(t) \right) \ \leq \ \frac{t^2}{(b t - \beta)} \varepsilon(t) \| x^* \|^2.
\]
Integrating the last inequality on $[T, t]$, where $T \geq t_0$, we deduce
\[
\frac{t}{b t - \beta} E_2(t) \ \leq \ \frac{T}{b T - \beta} E_2(T) + \| x^* \|^2 \int_T^t \frac{s^2}{(b s - \beta)} \varepsilon(s) ds.
\]
By the definition of $E_2$ we know
\[
 E_2(t) \ \geq \ (b t^2 - \beta t) \left( \Phi_{\lambda(t)} (x(t)) - \Phi^* \right).
\]
Combining these two inequalities, we deduce
\[
\Phi_{\lambda(t)} (x(t)) - \Phi^* \ \leq \ \frac{T}{t^2 (b T - \beta)} E_2(T) + \frac{\| x^* \|^2}{t^2} \int_T^t \frac{s^2}{(b s - \beta)} \varepsilon(s) ds.
\]
Now, 
\[
\lim_{t \to +\infty} \frac{T}{t^2 (b T - \beta)} E_2(T) \ = \ 0.
\]
Applying Lemma \ref{L_3} we deduce due to \eqref{integrability_T_0}
\[
\lim_{t \to +\infty} \frac{b t - \beta}{t^3} \int_T^t \frac{s^3}{(b s - \beta)} \frac{\varepsilon(s)}{s} ds \ = \ 0
\]
and thus
\[
\lim_{t \to +\infty} \frac{\| x^* \|^2}{t^2} \int_T^t \frac{s^2}{(b s - \beta)} \varepsilon(s) ds \ = \ 0.
\]
Therefore, we establish 
\[
\lim_{t \to +\infty} \Phi_{\lambda(t)} (x(t)) - \Phi^* \ = \ 0.
\]
Again, by the definition of the proximal mapping
\[
\Phi_{\lambda(t)}(x(t)) - \Phi^* = \ \Phi \left( \prox\nolimits_{\lambda(t) \Phi}(x(t)) \right) - \Phi^* + \frac{1}{2\lambda(t)} \| \prox\nolimits_{\lambda(t) \Phi} (x(t)) - x(t) \|^2 \quad \forall t \geq t_0.
\]
Using the fact that $\lambda$ is bounded for all $t \geq t_0$ we deduce
\[
\lim_{t \to +\infty} \| \prox\nolimits_{\lambda(t) \Phi} (x(t)) - x(t) \| \ = \ 0
\]
and
\[
\lim_{t \to +\infty} \Phi \left( \prox\nolimits_{\lambda(t) \Phi} (x(t)) \right) - \Phi^* \ = \ 0.
\]

\end{proof}

We are in position now to formulate the analogue of Theorem \ref{Str_conv}.
\begin{theorem}\label{Str_conv_1}

Suppose that $\lambda$ is bounded for all $t \geq t_0$, $b(t) \equiv b \geq \frac{1}{2} + \frac{\beta}{t_0}$ and \eqref{A_9} and \eqref{integrability_T_0} hold. Assume, in addition, that 
\begin{equation}\label{A_2_2}
    \lim_{t \to +\infty} t^2 \varepsilon(t) \ = \ +\infty,
\end{equation}
\begin{equation}\label{A_3_3}
    2 \beta t + \beta \dot \lambda(t) - b t \left( \dot \lambda(t) + 2 \beta \right) + 2 \beta^2 + \beta \ \leq \ 0 \text{ for all } t \geq t_0
\end{equation}
and
\begin{equation}\label{A_5_5}
    \lim_{t \to +\infty} \frac{\beta}{t^2 \varepsilon(t)} \int_{t_0}^t s^2 \varepsilon^2(s) ds \ = \ 0.
\end{equation}
If $x : [t_0,+\infty) \mapsto H$ is a solution to \eqref{Syst} and the trajectory $x(t)$ stays either inside or outside the ball $B(0, \| x^* \|)$, then $x(t)$ converges to minimal norm solution $x^* = \proj_{\argmin \Phi}(0)$, as $t \to +\infty$. Otherwise, $\liminf_{t \to +\infty} \| x(t) - x^* \| = 0$.

\end{theorem}

\begin{proof}

The proof goes in line with the one of Theorem \ref{Str_conv} by taking $\alpha = 3$, $b(t) \equiv b > 0$, $q = 2$, $p = 0$ and referring to Theorem \ref{Str_conv_aux_th_0} instead of Theorem \ref{Str_conv_aux_th} in the second and third cases.

\end{proof}

\end{subsection}

\begin{section}{Analysis of the conditions}

Since all the conditions cannot be satisfied simultaneously, let us treat them separately, namely:
\begin{enumerate}
    
    \item In order to obtain the fast convergence rates of the function values we require that for all $t \geq t_0$:
    
    \begin{itemize}
    
        \item $\alpha \ > \ 3$;
        
        \item the existence of $a \geq 1$ such that $ 2 \dot \varepsilon(t) \ \leq \ - a \beta \varepsilon^2(t) $,

        \item $ b(t_0) \ \geq \ \frac{\beta}{t_0} \text{ and } b(t_0) > \frac{1}{a} $;
        
        \item $ \int_{t_0}^{+\infty} t \varepsilon(t) dt \ < \ +\infty $

        and 
        
        \item the existence of $0 < \delta < \alpha - 3$ such that $ (\alpha - 3) t b(t) - t^2 \dot b(t) + \beta (2 - \alpha) \ \geq \ \delta t b(t) $.
    
    \end{itemize}

    \item For the strong convergence of the trajectories we require the following for all $t \geq t_0$:
    
    \begin{itemize}
    
        \item $\alpha \ > \ 3$;
        
        \item $\lambda$ is bounded;

        \item $ \frac{\alpha - 3}{3} b(t) - t \dot b(t) + \frac{\alpha \beta}{3} \ \geq \ 0 $;
        
        \item $ (\alpha - 3) t b(t) - t^2 \dot b(t) + \beta (2 - \alpha) \ \geq \ 0 $;

        \item  the existence of $a \geq 1$ such that $ 2 \dot \varepsilon(t) \ \leq \ - a \beta \varepsilon^2(t) $, $ b(t_0) > \frac{1}{a} \text{ and } b(t_0) \geq \frac{1}{2} + \frac{\beta}{t_0} $;
    
        \item $ \int_{t_0}^{+\infty} \frac{\varepsilon(t)}{t b(t)} dt \ < \ +\infty $;
    
        \item $ 2 \alpha (\alpha - 3) - 9 t^2 \varepsilon(t) + 6 \alpha \beta \ \leq \ 0 $;
    
        \item $ 18 \beta t + 9 \beta \dot \lambda(t) - 9 t b(t) \left( \dot \lambda(t) + 2 \beta \right) + 3 (\alpha + 3) \beta^2 + \alpha^2 \beta  \ \leq \ 0 $;
    
        \item $ \lim_{t \to +\infty} \frac{\beta}{t^{\frac{\alpha}{3} + 1} \varepsilon(t)} \int_{t_0}^t s^{\frac{\alpha}{3} + 1} \varepsilon^2(s) ds \ = \ 0 $.

    \end{itemize}

\end{enumerate}

We will analyse these conditions in details for the polynomial choice of functions $b$ and $\varepsilon$, namely, $b(t) = b t^n$ and $\varepsilon(t) = \frac{\varepsilon}{t^d}$, where $b$ is positive, $n \geq 0$ and $\varepsilon, d > 0$.

\begin{subsection}{Setting for the fast convergence rates of the function values}

The set of the conditions becomes for all $t \geq t_0$

\begin{enumerate}
    
        \item $\alpha \ > \ 3$;
        
        \item  there exists $a \geq 1$ such that $ -\frac{2 d \varepsilon}{t^{d+1}}  \ \leq \ - \frac{a \beta \varepsilon^2}{t^{2d}} $,

        \item $ b(t_0) \ \geq \ \frac{\beta}{t_0} \text{ and } b(t_0) > \frac{1}{a} $;
        
        \item $ \int_{t_0}^{+\infty} \frac{\varepsilon}{t^{d-1}} dt \ < \ +\infty $

        and 
        
        \item there exists $0 < \delta < \alpha - 3$ such that $ (\alpha - 3) b t^{n+1} - b n t^{n+1} + \beta (2 - \alpha) \ \geq \ \delta b t^{n+1} $.
    
\end{enumerate}

After some simple algebraic computations one may discover that in order to satisfy all the conditions at the same time it is enough to assume 
\[
\alpha - 3 \ > \ n \ \geq \ 0 \text{ (condition $5$) }
\]
and
\[
d \ > \ 2 \text{ and } d \ \geq \ \frac{\beta \varepsilon}{2} \text{ (conditions $2$ and $4$) },
\]
since all the other inequalities could be fulfilled by taking the appropriate $t_0$, namely,
\[
t_0 \ \geq \ \max \left\{ \sqrt[n+1]{\frac{\beta}{b}}, \sqrt[n+1]{\frac{\beta (\alpha - 2)}{b (\alpha - 3 - n)}} \right\} \text{ and } t_0 \ > \ \frac{1}{\sqrt[n]{b}}.
\]

\end{subsection}

\begin{subsection}{Setting for the strong convergence of the trajectories}

The set of the conditions becomes for all $t \geq t_0$

\begin{enumerate}
    
        \item $\alpha \ > \ 3$;
        
        \item $\lambda$ is bounded;

        \item $ \frac{\alpha - 3}{3} b t^{n+1} - b n t^{n+1} + \frac{\alpha \beta}{3} \ \geq \ 0 $;
        
        \item $ (\alpha - 3) b t^{n+1} - b n t^{n+1} + \beta (2 - \alpha) \ \geq \ 0 $;

        \item there exists $a \geq 1$ such that $ -\frac{2 d \varepsilon}{t^{d+1}}  \ \leq \ - \frac{a \beta \varepsilon^2}{t^{2d}} $, $ b(t_0) > \frac{1}{a} \text{ and } b(t_0) \geq \frac{1}{2} + \frac{\beta}{t_0} $;
    
        \item $ \int_{t_0}^{+\infty} \frac{\varepsilon}{b t^{n+d+1}} dt \ < \ +\infty $;
    
        \item $ 2 \alpha (\alpha - 3) - \frac{9 \varepsilon}{t^{d-2}} + 6 \alpha \beta \ \leq \ 0 $;
    
        \item $ 18 \beta t + 9 \beta \dot \lambda(t) - 9 b t^{n+1} \left( \dot \lambda(t) + 2 \beta \right) + 3(\alpha + 3) \beta^2 + \alpha^2 \beta  \ \leq \ 0 $;
    
        \item $ \lim_{t \to +\infty} \frac{\beta}{\varepsilon t^{\frac{\alpha}{3} - d + 1}} \int_{t_0}^t \varepsilon^2 s^{\frac{\alpha}{3} - 2d + 1} ds \ = \ 0 $.

\end{enumerate}

Again, analysis of the set of conditions leads to the following conclusion: 

\begin{itemize}
        
        \item $\lambda$ is bounded \text{ (condition $2$) };
        
        \item $0 \ \leq \ n \ \leq \ \frac{\alpha - 3}{3}$ and $ \alpha \ > \ 3$ \text{ (condition $3$) }; 
        
        \item $\max \left\{1, \frac{\beta \varepsilon}{2} \right\} \ \leq \ d \ \leq \ 2$ \text{ (conditions $5$, $7$, $8$, $9$) }.
        
\end{itemize}

As before, $t_0$ should be chosen appropriately.


\end{subsection}

\begin{subsection}{The case $\alpha = 3$}

In this case the following has to be assumed: there exists $a \geq 1$ such that for all $t \geq t_0$

\begin{enumerate}
    
    \item $ \lambda(t) $ is bounded;
    
    \item $ 2 \dot \varepsilon(t) \ \leq \ - a \beta \varepsilon^2(t), \ b \ > \ \frac{1}{a}$ and $b \geq \frac{1}{2} + \frac{\beta}{t_0} $;
    
    \item $ \int_{t_0}^{+\infty} \frac{\varepsilon(t)}{t} dt \ < \ +\infty $;
    
    \item $ \lim_{t \to +\infty} t^2 \varepsilon(t) \ = \ +\infty $;
    
    \item $ 2 \beta t + \beta \dot \lambda(t) - b t \left( \dot \lambda(t) + 2 \beta \right) + 2 \beta^2 + \beta \ \leq \ 0 $;
    
    \item $ \lim_{t \to +\infty} \frac{\beta}{t^2 \varepsilon(t)} \int_{t_0}^t s^2 \varepsilon^2(s) ds \ = \ 0 $.

\end{enumerate}

Essentially, for the polynomial choice of parameters that means $b \ \geq \ 1$ and

\begin{itemize}
    
    \item $ \lambda(t) $ is bounded \text{ (condition $2$) };
    
    \item $ \max \left\{1, \frac{\beta \varepsilon}{2} \right\} \ \leq \ d \ < \ 2 $ \text{ (conditions $4$, $5$, $6$) },

\end{itemize}

so with the appropriate choice of $t_0$ the whole set of conditions is fulfilled.


\end{subsection}

\end{section}

\begin{section}{Numerical examples}

\begin{subsection}{The rates of convergence of the Moreau envelope values}

Consider the objective function $\Phi: \mathbb{R} \to \mathbb{R}$, $\Phi(x) = |x| + \frac{x^2}{2}$ and let us plot the values of its Moreau envelope as well as the gradient of its Moreau envelope for different polynomial functions $\lambda$, $\varepsilon$ and $b$ to illustrate the theoretical results with some numerical examples. We take $\lambda(t) = t^l$, $\varepsilon(t) = \frac{1}{t^d}$, $b(t) = t^n$ with $x(t_0) = x_0 = 10$, $\dot x(t_0) = 0$, $\alpha = 10$ and $t_0 = 1.4$.

First, let us take different time scaling parameter $b$ with $l = 0$ and $d = 3$ and see how it affects the behaviour of the system \eqref{Syst} (see figure 1):

\begin{figure}[H]
     \centering
     \begin{subfigure}[b]{0.49\textwidth}
         \centering
         \includegraphics[width=\textwidth]{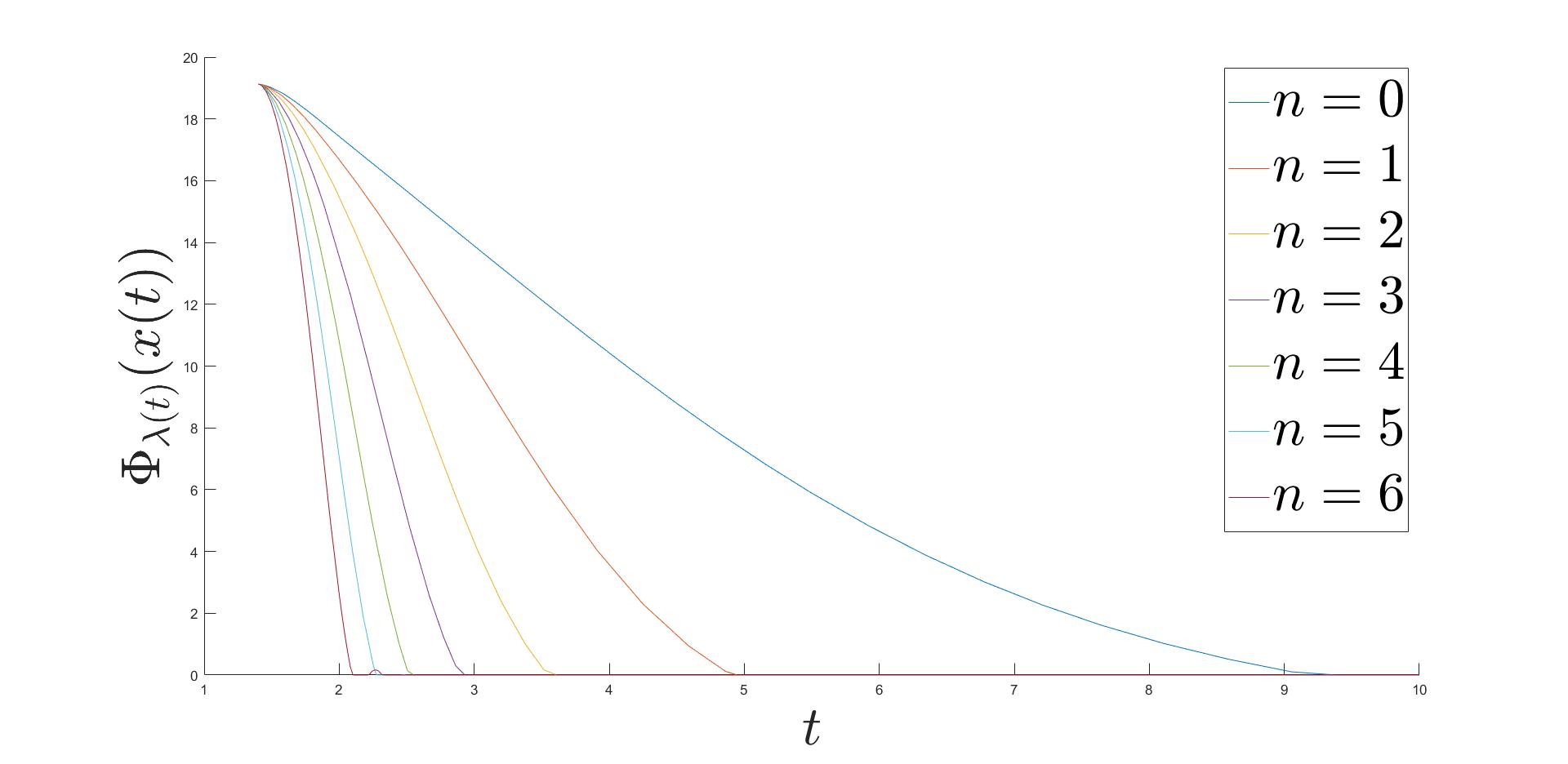}
         \caption{Moreau envelope values}
     \end{subfigure}
     \hfill
     \begin{subfigure}[b]{0.49\textwidth}
         \centering
         \includegraphics[width=\textwidth]{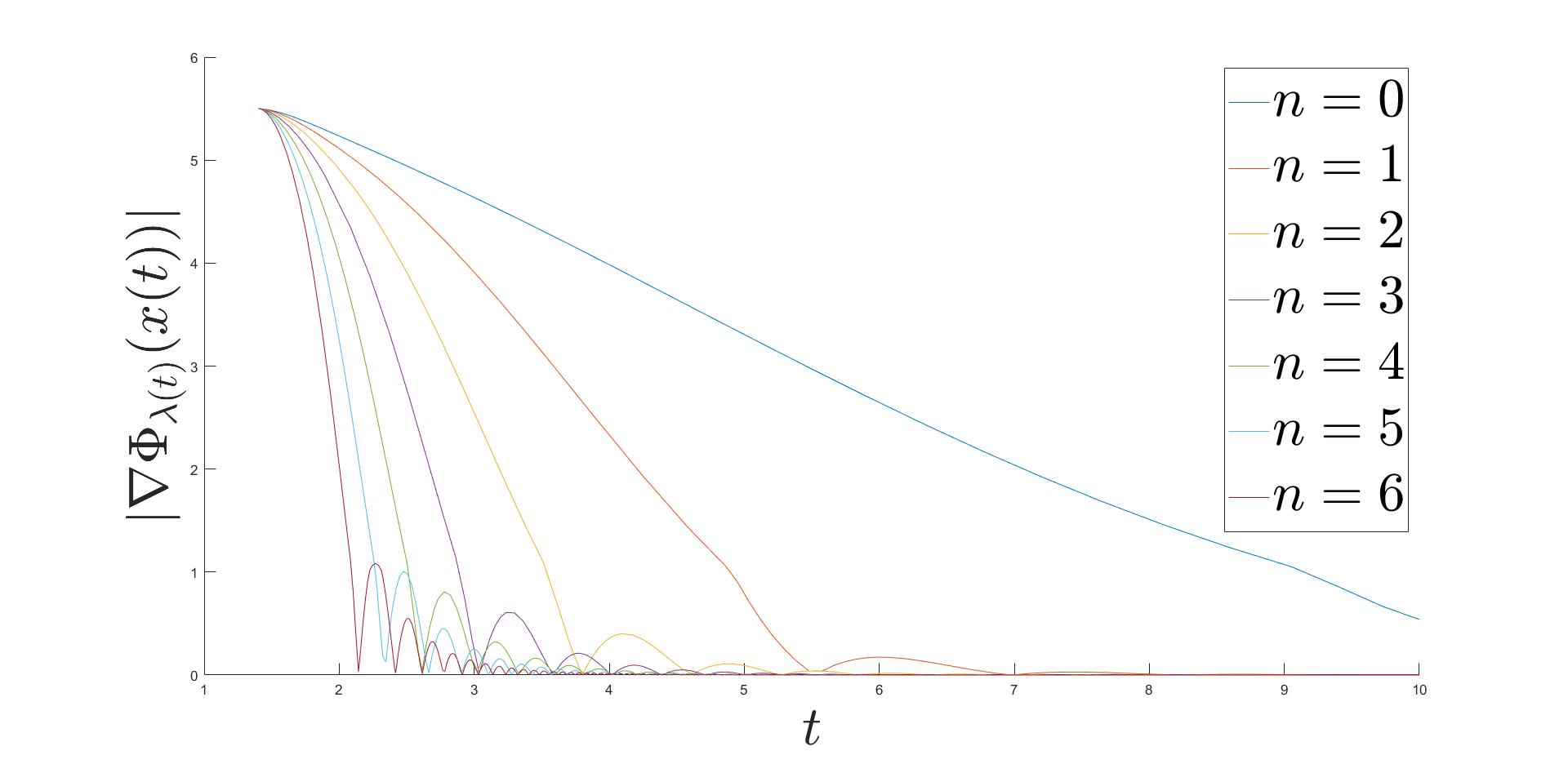}
         \caption{Gradient}
     \end{subfigure}
     \hfill
     \caption{$l = 0$ and $d = 3$.}
\end{figure}

As expected, the faster $b$ grows, the faster the convergence is.

Consider now different Moreau envelope parameter $\lambda$ with $d = 3$ and $n = 0$ (see figure 2):

\begin{figure}[H]
     \centering
     \begin{subfigure}[b]{0.49\textwidth}
         \centering
         \includegraphics[width=\textwidth]{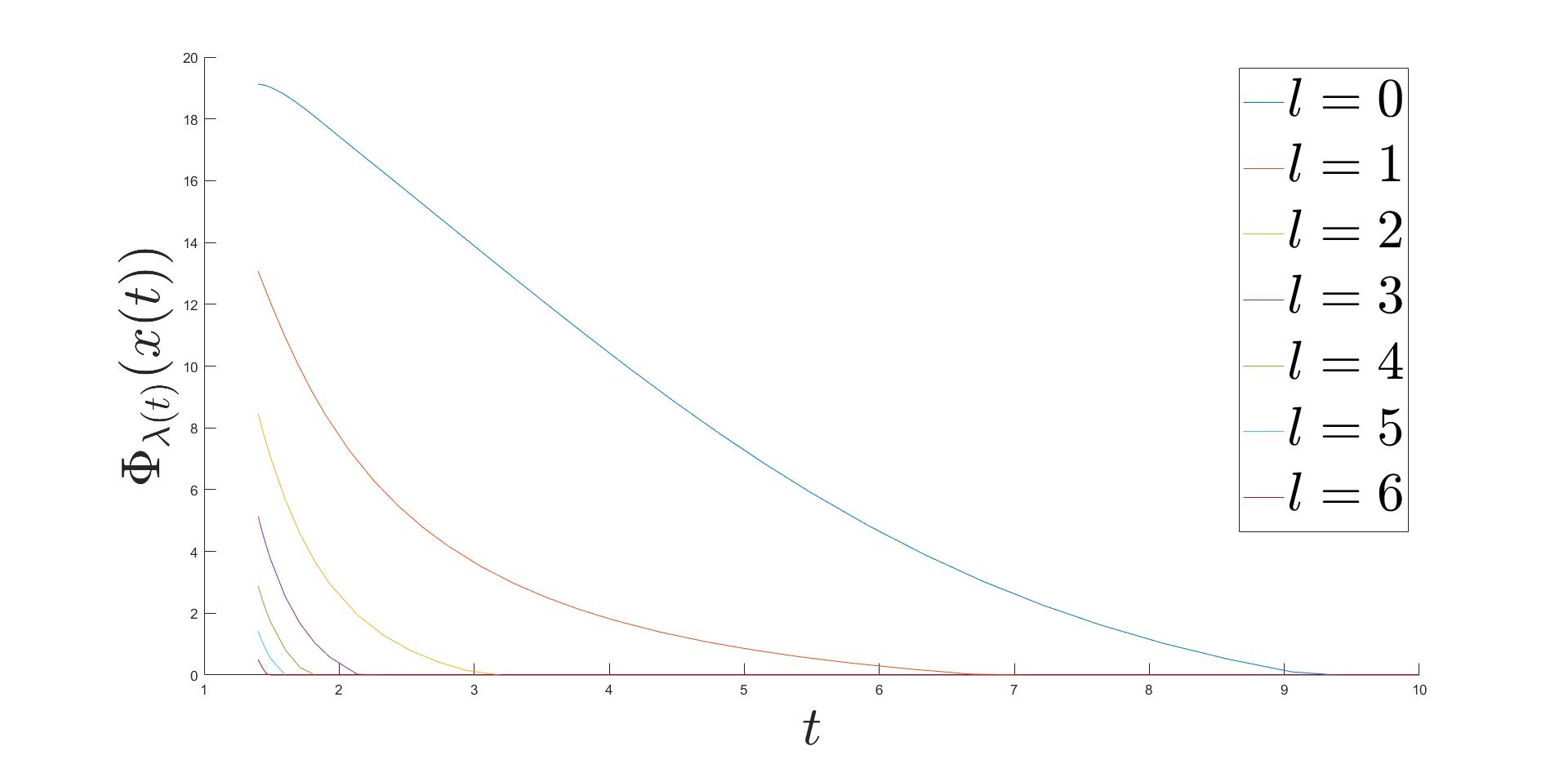}
         \caption{Moreau envelope values}
     \end{subfigure}
     \hfill
     \begin{subfigure}[b]{0.49\textwidth}
         \centering
         \includegraphics[width=\textwidth]{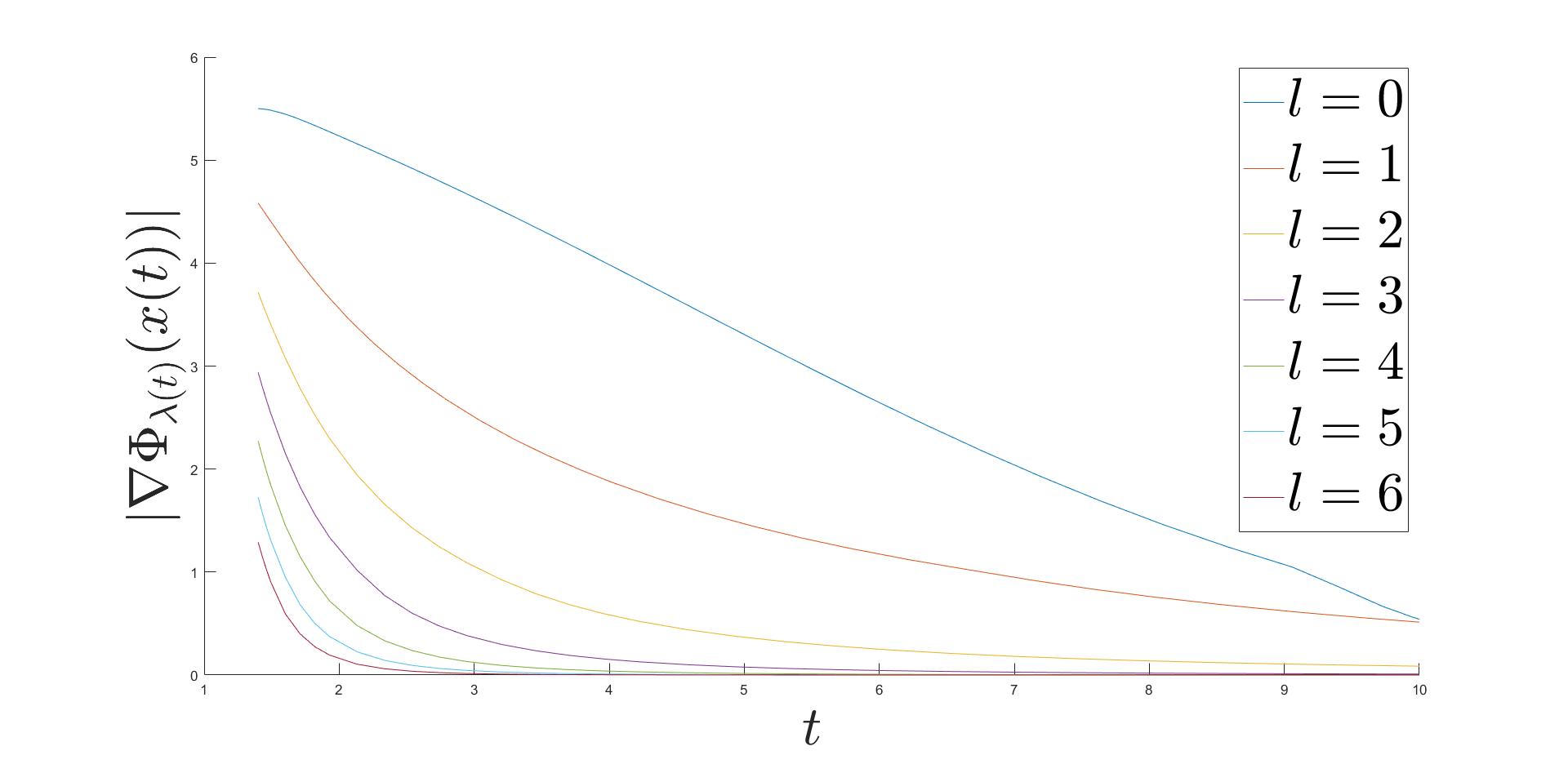}
         \caption{Gradient}
     \end{subfigure}
     \hfill
     \caption{$d = 3$ and $n = 0$.}
\end{figure}

Note that the difference in the starting point comes from the fact that $t_0 \neq 1$, and for different exponents $l$ the value $t_0^l$ is also different. As predicted by theory, a faster growing function $\lambda$ leads to faster convergence of not only the gradient of Moreau envelope of the objective function $\Phi$, but also of the values of the Moreau envelope themselves.

Varying the Tikhonov function $\varepsilon$ for $n = 0$ and $l = 0$ does not affect the system, which is illustrated by the following plot (see figure 3):

\begin{figure}[H]
     \centering
     \begin{subfigure}[b]{0.49\textwidth}
         \centering
         \includegraphics[width=\textwidth]{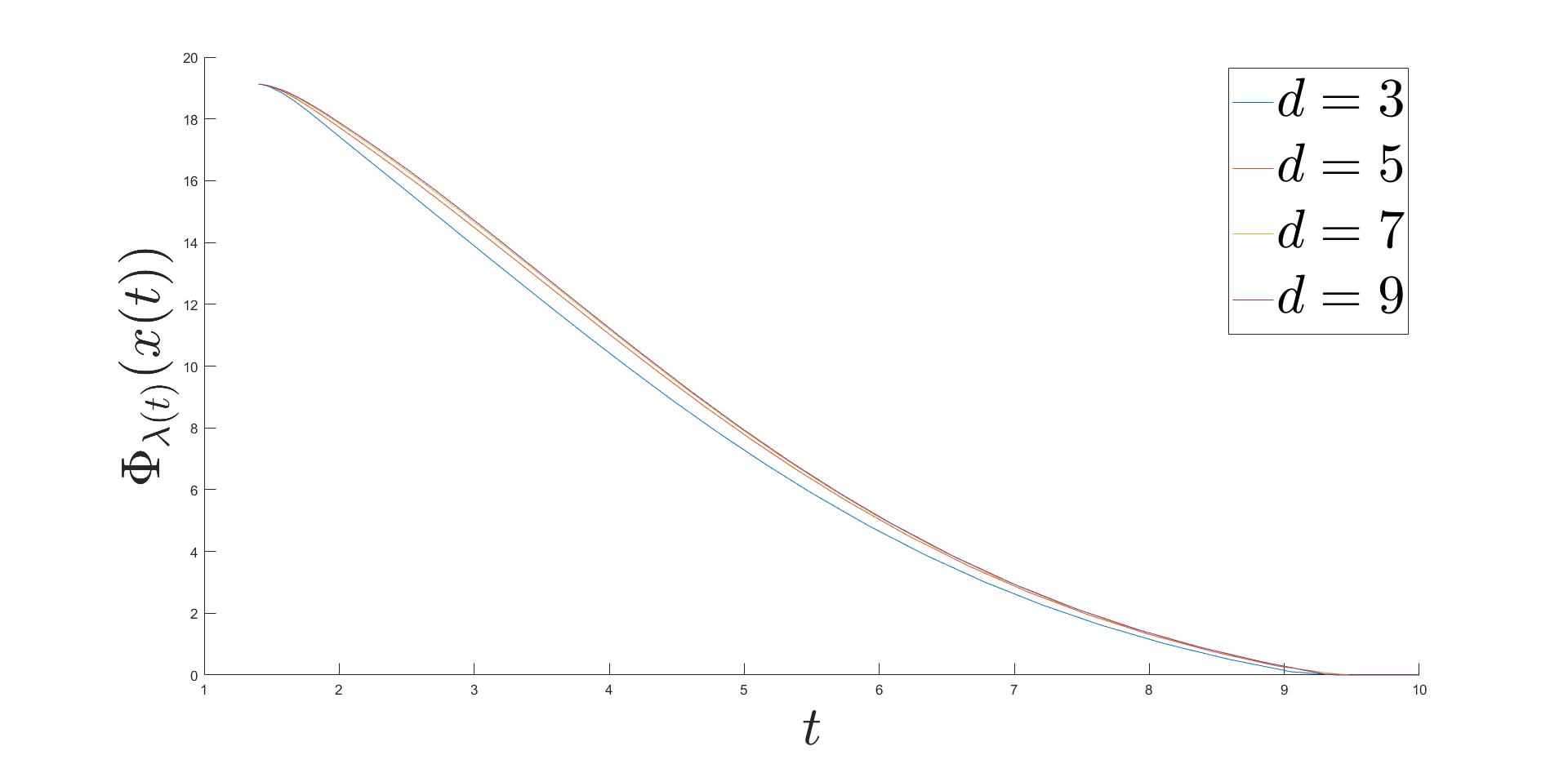}
         \caption{Moreau envelope values}
     \end{subfigure}
     \hfill
     \begin{subfigure}[b]{0.49\textwidth}
         \centering
         \includegraphics[width=\textwidth]{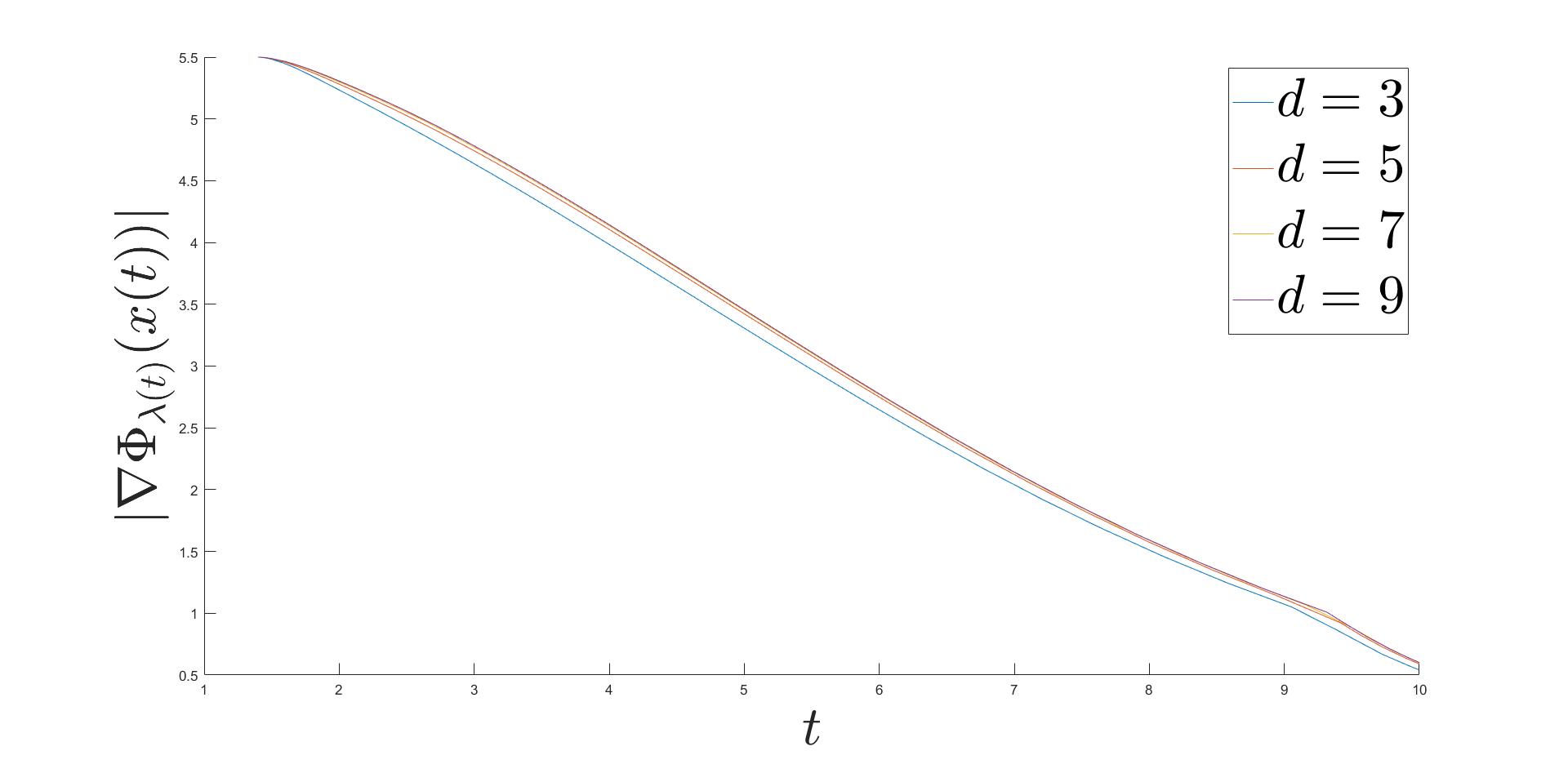}
         \caption{Gradient}
     \end{subfigure}
     \hfill
     \caption{$n = 0$ and $l = 0$.}
\end{figure}

\end{subsection}

\begin{subsection}{Strong convergence of the trajectories}

For a different objective function let us investigate the strong convergence of the trajectories of \eqref{Syst}:
\begin{equation*}
\Phi(x) \ = \ 
\begin{cases}
    &|x - 1|, \ x > 1 \\
    &0, \ x \in [-1, 1] \\
    &|x + 1|, \ x < -1.
\end{cases}
\end{equation*}
The set $\argmin \Phi$ is nothing but the segment $[-1, 1]$ and $0$ is its element of minimal norm. Let us fix $\alpha = 6$ and $n = 0.7$. First we take constant lambda ($\lambda(t) = 1$ for all $t \geq t_0$) and plot the behaviour of the trajectories of \eqref{Syst} with and without Tikhonov term (see figure 4).

\begin{figure}[H]
     \centering
     \begin{subfigure}[b]{0.49\textwidth}
         \centering
         \includegraphics[width=\textwidth]{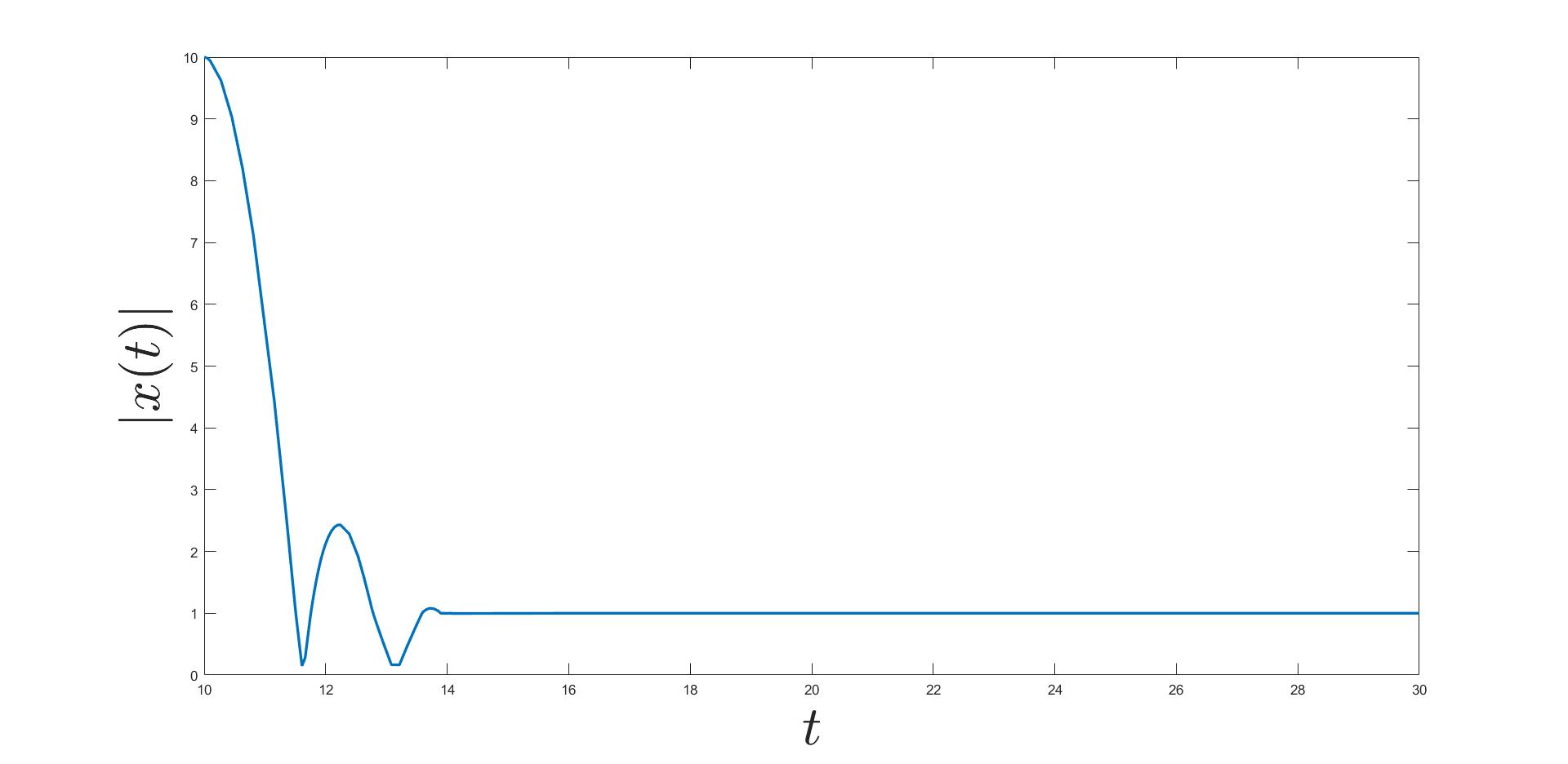}
         \caption{$\varepsilon(t) = 0$}
     \end{subfigure}
     \hfill
     \begin{subfigure}[b]{0.49\textwidth}
         \centering
         \includegraphics[width=\textwidth]{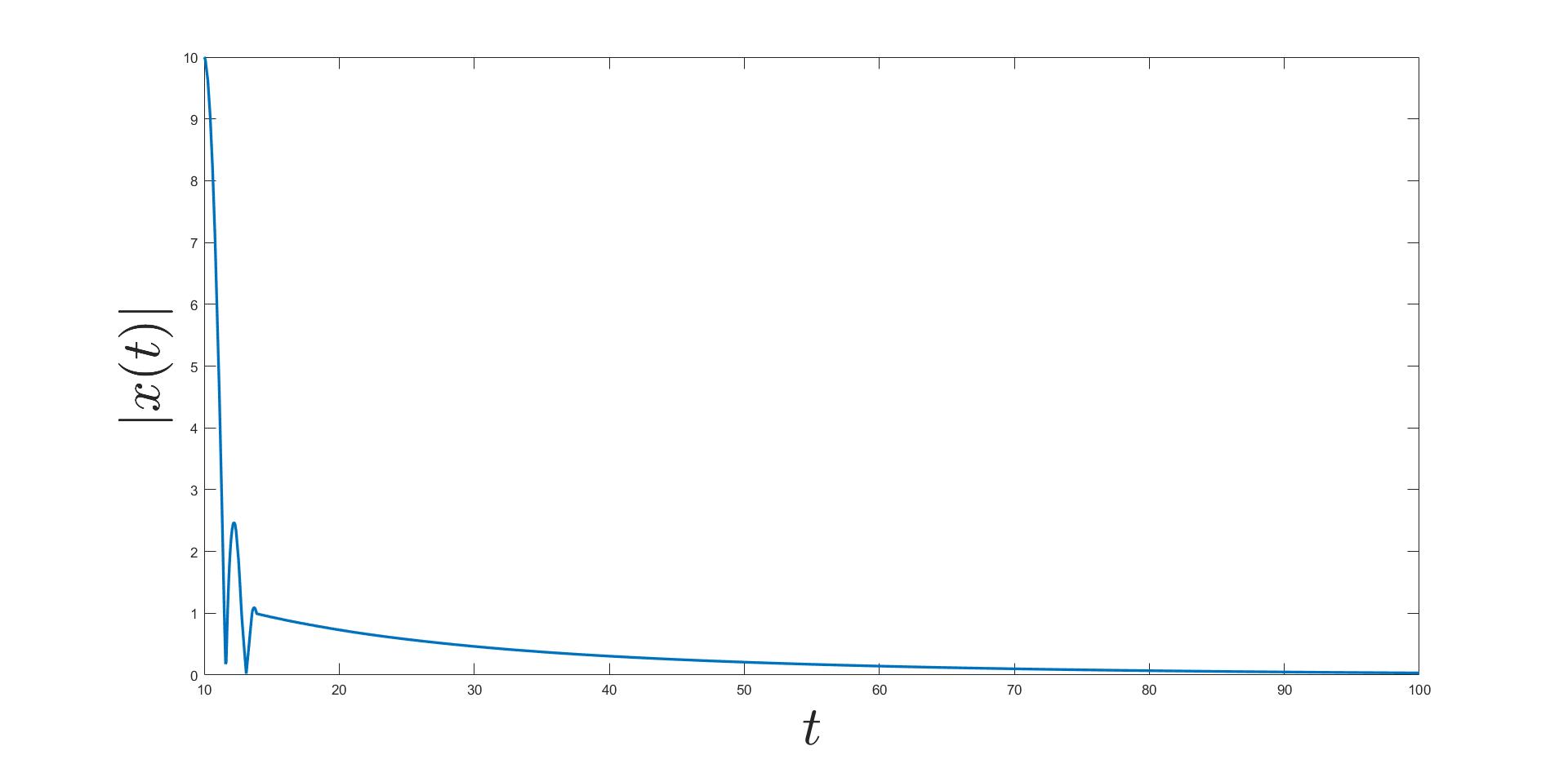}
         \caption{$\varepsilon(t) = \frac{1}{t^\frac{3}{2}}$}
     \end{subfigure}
     \hfill
     \caption{The role of the Tikhonov term.}
\end{figure}

As we see in case there is no Tikhonov regularization the trajectories converge to the minimizer $1$ of $\Phi$, but the Tikhonov term actually guarantees the convergence towards the minimal norm solution, which is $0$.

Another comparison was made for non-constant lambda: $\lambda(t) = 1 - \frac{1}{t^l}$ for $l = 1$ (for different $l$'s the picture is the same), illustrating similar behaviour (see figure 5).

\begin{figure}[H]
     \centering
     \begin{subfigure}[b]{0.49\textwidth}
         \centering
         \includegraphics[width=\textwidth]{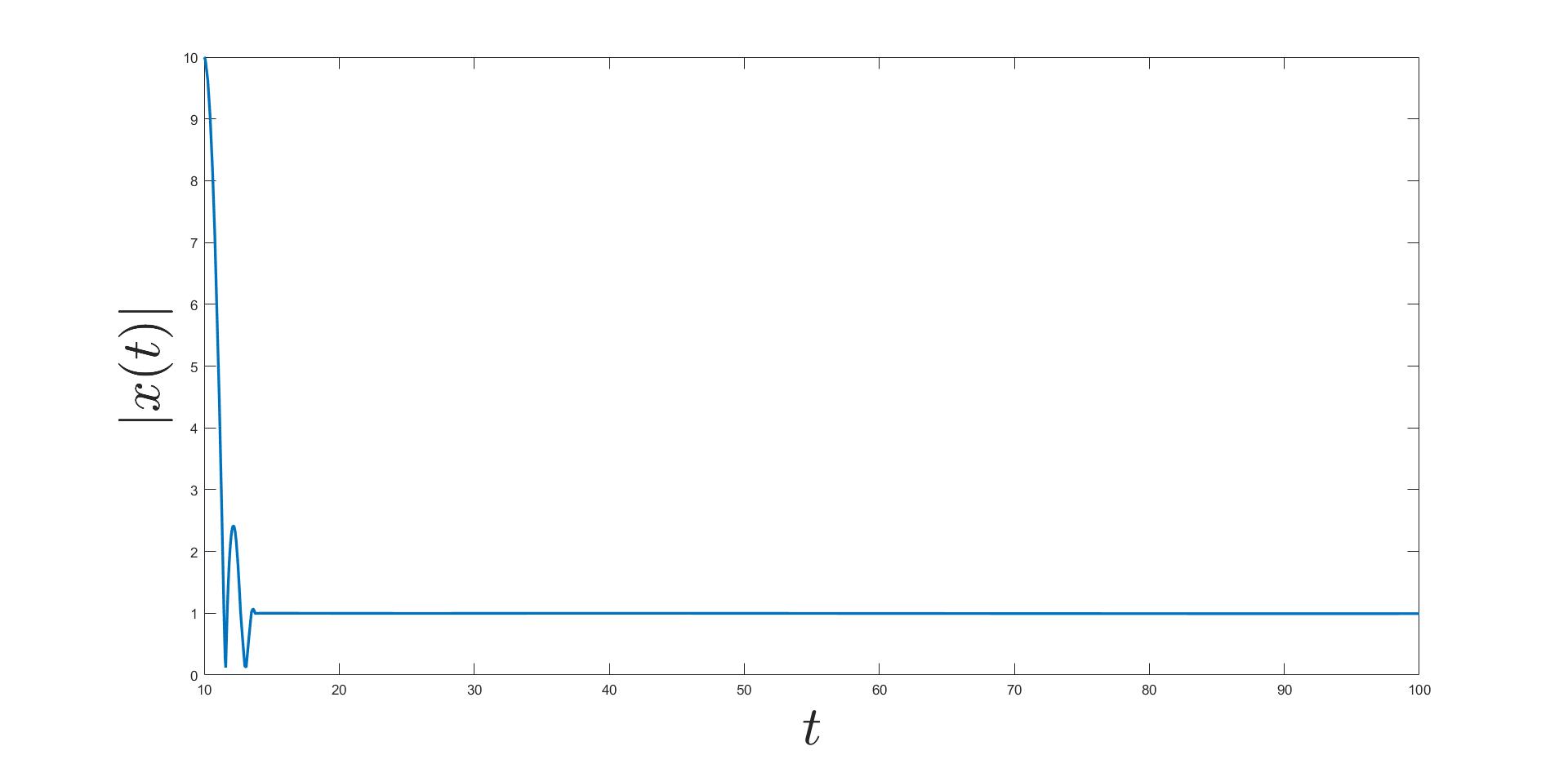}
         \caption{$\varepsilon(t) = 0$}
     \end{subfigure}
     \hfill
     \begin{subfigure}[b]{0.49\textwidth}
         \centering
         \includegraphics[width=\textwidth]{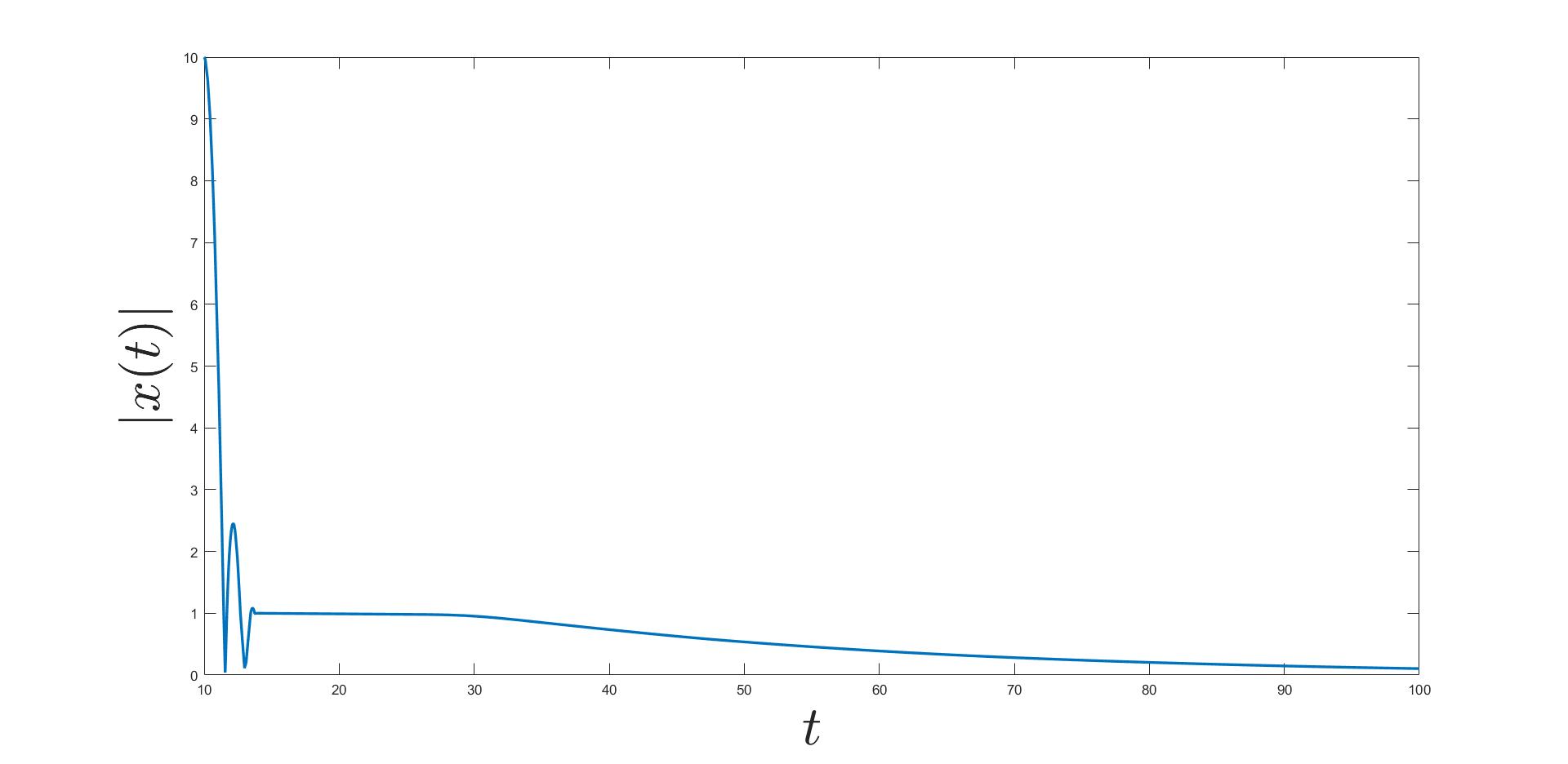}
         \caption{$\varepsilon(t) = \frac{1}{t^\frac{3}{2}}$}
     \end{subfigure}
     \hfill
     \caption{The role of the Tikhonov term.}
\end{figure}

Finally, for the same choice of $\lambda$ let us take different Tikhonov terms to figure out how changing them affects the trajectories of \eqref{Syst} (see figure 6):

\begin{figure}[H]
    \includegraphics[width=\textwidth]{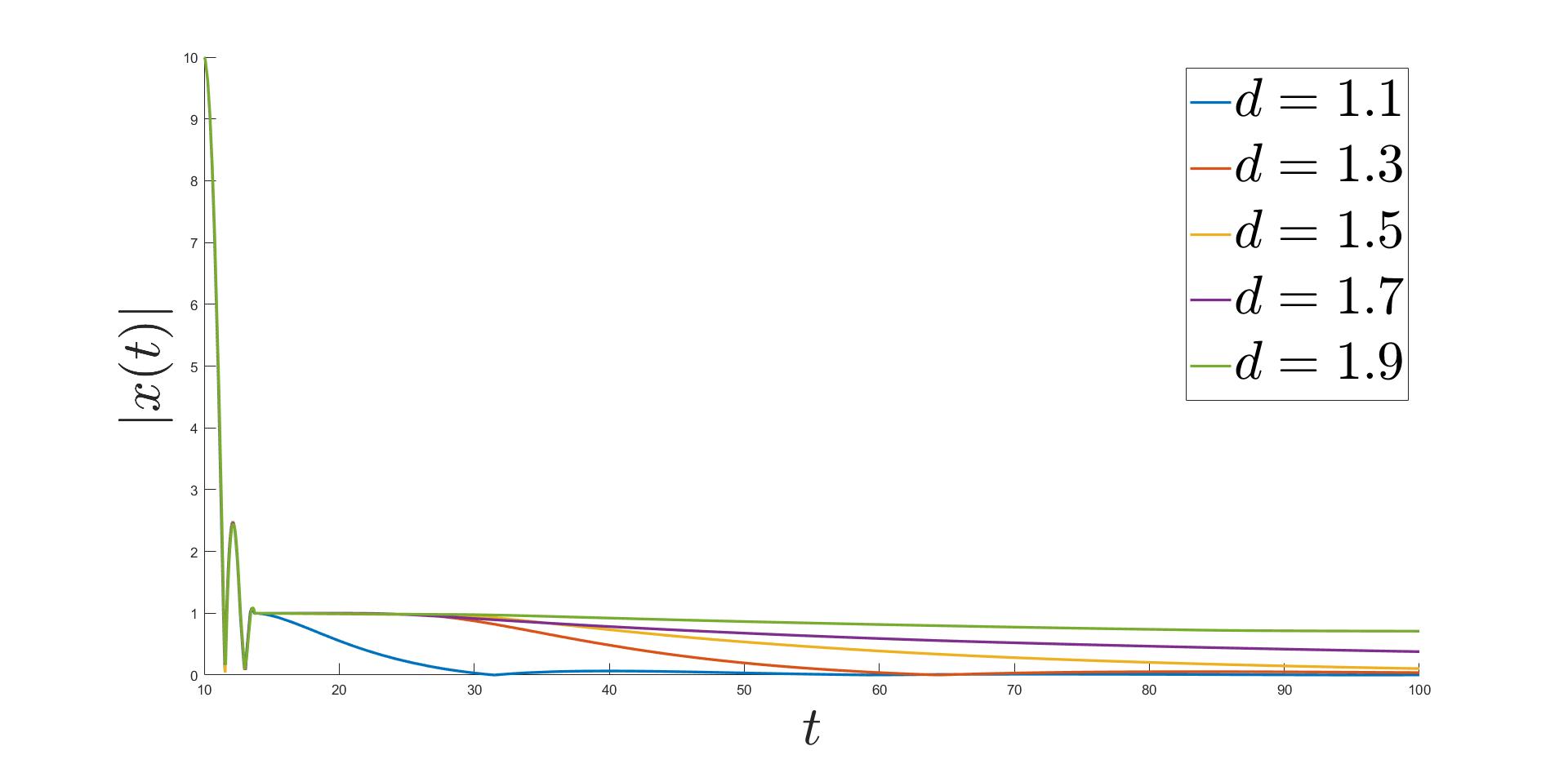}
    \caption{$n = 0.7$ and $\lambda(t) = 1 - \frac{1}{t}$.}
\end{figure}

We see, that the faster $\varepsilon$ decays, the slower trajectories converge.

\end{subsection}

\end{section}

\begin{section}{Appendix}

Let us state here some auxiliary lemmas which we used in our analysis. For the proof of the following lemma we refer to \cite{AAS}.

\begin{lemma}\label{L_1}
Suppose that $f: [t_0, +\infty) \to \mathbb{R}$ is locally absolutely continuous and bounded from below and there exists $g \in L^1([t_0, +\infty), \mathbb{R})$ such that for almost all $t \geq t_0$
\[
\frac{d}{dt} f(t) \ \leq \ g(t).
\]
Then there exists $\lim_{t \to +\infty} f(t) \in \mathbb{R}$.
\end{lemma}

For the proof of the next lemma we refer to \cite{APR}.
\begin{lemma}\label{L_2}
Let $H$ be a real Hilbert space and $x : [t_0, +\infty) \mapsto \mathbb{H}$ be a continuously differentiable function satisfying $ x(t) + \frac{t}{\alpha} \dot x(t) \to \ L $ as $ t \to +\infty $, with $ \alpha > 0 $ and $ L \in \mathbb{H} $. Then $ x(t) \to L $ as $ t \to +\infty $.
\end{lemma}

For the proof of the final Lemma we refer to \cite{ACR}.
\begin{lemma}\label{L_3}
Let $\delta > 0$ and $f \in L^1 \left( (\delta, +\infty), \mathbb{R} \right)$ be a non-negative and continuous function. Let $g: [\delta, +\infty) \to [0, +\infty)$ be a non-decreasing function such that $\lim_{t \to +\infty} g(t) \ = \ +\infty$. Then it holds
\[
\lim_{t \to +\infty} \frac{1}{g(t)} \int_\delta^t g(s) f(s) ds \ = \ 0.
\]
\end{lemma}

\end{section}

\section{Declarations}

\begin{center}
    \textbf{Acknowledgements}
\end{center}

The authors are grateful to two anonymous reviewers for their remarks on this manuscript and for meaningful suggestions, which improved the quality of this paper.

\begin{center}
    \textbf{Funding and/or Conflicts of interests/Competing interests}
\end{center}

The authors declare that they have no competing interests subject to the topic of this article.

\begin{center}
    \textbf{Data availability}
\end{center}

Data sharing not applicable to this article as no datasets were generated or analysed during the current study.

\end{document}